\documentclass[12pt]{article}
\usepackage[utf8]{inputenc}
\usepackage[margin=1in]{geometry}
\usepackage{enumerate}
\usepackage{graphicx,float}
\usepackage{amssymb,MnSymbol,dsfont}
\usepackage{doi}
\usepackage{color}
\usepackage{hyperref}
\usepackage{cleveref}
\usepackage{algorithm}
\usepackage{algpseudocode}
\usepackage{subcaption}
\usepackage[numbers]{natbib}
\usepackage{booktabs}

\newcommand{\RR}{\mathds{R}}
\newcommand{\E}{\mathbb{E}}
\newcommand{\N}{\mathcal{N}}

\newcommand{\Mc}{\mathcal{M}}
\newcommand{\nhat}{\hat{n}}

\newcommand{\Var}{\mathrm{Var}}
\newcommand{\myrightarrow}[1]{\xrightarrow{\makebox[1em][c]{$\scriptstyle#1$}}}
\newcommand{\Rcancel}{R}
\newcommand{\Rgauss}{R_{\mathrm{Gauss}}}
\newcommand{\Rprop}{R_{\mathrm{prop}}}

\newcommand{\hess}{\mathrm{Hess}}

\newtheorem{thm}{Theorem}
\newtheorem{lem}{Lemma}

\newtheorem{rmk}{Remark}
\newtheorem{prop}{Proposition}
\newtheorem{cor}{Corollary}

\newenvironment{pf}{\noindent\textbf{Proof.}}{\hfill$\square$}

\DeclareUnicodeCharacter{2212}{\ensuremath{-}}

\algrenewcommand\algorithmiccomment[1]{\hfill // #1}

\title{Optimal scaling and curvature correction for constrained MCMC on spheres
\thanks{For the purpose of open access, the author has applied a Creative Commons Attribution (CC BY) licence to any Author Accepted Manuscript version arising from this submission.}
}

\date{}

\author{Nayef Shkeir\\
\small Mathematics Institute, University of Warwick, Coventry, CV4 7AL, United Kingdom\\
\and Tobias Grafke\\
\small Mathematics Institute, University of Warwick, Coventry, CV4 7AL, United Kingdom}

\begin{document}

\maketitle

\begin{abstract}
Sampling Gaussian measures restricted to nonlinear manifolds in high dimensions requires a proposal kernel that respects both geometric constraints and dimensional scaling requirements. We introduce tangent preconditioned Crank-Nicolson (pCN), which combines a manifold projection with a cancellation of the ambient Gaussian contribution, leaving only a geometric residual. Considering a shifted sphere with $\sqrt d$ scaling of radius and centre as a tractable benchmark embedded manifold, we prove that uncorrected tangent pCN results in a proposal noise scale of order $d^{-1/6}$ (compared to $d^{-1/2}$ for a tangent random-walk). We establish asymptotic optimal tuning for the geodesic expected squared jump distance which is maximised at an acceptance probability of 0.574 for the family of drift-corrected tangent proposals containing tangent pCN and tangent MALA. We also prove a joint diffusion limit for the parallel and orthogonal coordinates to the sphere's centre vector of tangent pCN. The orthogonal diffusion speed matches the limiting geodesic expected squared jump distance and is maximised at acceptance probability 0.574, whereas the parallel speed is maximised at 0.712, demonstrating that optimal tuning depends on the observable. We then identify an intrinsic precision correction characterised by potential Hessian and Ricci curvature that yields a proposal kernel with dimension-independent fixed step size limit. We discuss a generalisation to arbitrary hypersurfaces and present nondegenerate scaling of tangent pCN applied to conditioned diffusions under resolution refinement.
\end{abstract}
\section{Introduction}

Nonlinear constraints arise naturally when scientific models are subject to physical constraints or observations: examples range from classical molecular dynamics~\cite{frenkel-smit:2002, lelievre-rousset-stoltz:2012}, over inverse problems~\cite{stuart:2010, beskos-girolami-lan-etal:2017} to modern tools such as generative AI~\cite{ho-jain-abbeel:2020, song-sohl-dickstein-kingma-etal:2021}. In the context of sampling from a Gaussian model, such constraints confine the ambient distribution onto a lower-dimensional submanifold, and a sampler needs to accommodate both the manifold geometry and the target measure. These goals become particularly hard to satisfy as the dimension of the problem grows.

For unconstrained problems, high-dimensional samplers are constructed by preserving a Gaussian reference measure, with preconditioned Crank-Nicolson (pCN)~\cite{cotter-roberts-stuart-etal:2013} as prime example. For target measures with a density relative to that reference, the Gaussian contribution fully cancels from the Metropolis-Hastings adjustment, leading to dimension-independent scaling of the resulting algorithm.

A nonlinear constraint breaks this cancellation, as ambient pCN proposals leave the submanifold, and enforcing the constraint modifies the finely-adjusted proposal law, leading to geometric residual terms. In this paper, we address the question of how much of the Gaussian cancellation can be retained, and how the geometric residual determines the scaling of the required step size with dimension. There are many approaches to sampling from manifolds~\cite{byrne-girolami:2013, diaconis-holmes-shahshahani:2013}, with initial developments often motivated by physical constraints (e.g.~molecular dynamics~~\cite{hartmann:2008}), or probabilistic and Bayesian conditioning or inference~\cite{stuart:2010}. For a generic manifold, the natural approaches are either to restrict proposal dynamics to the manifold~\cite{byrne-girolami:2013}, or to propose locally in tangent space, and then project back onto the manifold~\cite{zappa-holmes-cerfon-goodman:2018}. These strategies necessarily introduce additional structure, such as nontrivial volume forms, tilted tangent spaces and curvature. Hence, most literature discussing sampling on manifolds~\cite{girolami-calderhead:2011, lelievre-rousset-stoltz:2019} do not consider the high-dimensional limit or optimise for well-behaved dimensional scaling, and it is fair to say that dimensionally independent sampling on arbitrary manifolds is an unsolved problem. It is important to consider the contribution of some recent work on constructing dimension robust samplers on high-dimensional spheres. \citet{lie-rudolf-sprungk-etal:2023} discuss targets with a (bounded) density relative to an angular central Gaussian (ACG) measure and the dimension-independent spectral gap results require uniform bounds relative to this ACG reference that our concentrating target family does not satisfy. The geodesic slice samplers of \citet{habeck-hasenpflug-kodgirwar-etal:2025} apply to this family, but use a different proposal mechanism and their general convergence bounds depend on dimension. This work in contrast aims to study how the acceptance probability of a proposal that uses only a tangent step and a normal projection (and therefore applies to any hypersurface given some conditions) scales with dimension in the sense of the optimal scaling theory of \cite{roberts-gelman-gilks:1997,roberts-rosenthal:1998}.

To this end, we here introduce tangent pCN, combining a Gaussian tangent space proposal with normal projection and Metropolis-Hastings correction. We show its geometrically corrected cancellation properties to express the remaining Gaussian contribution through a normal term and curvature effects, which allows us to study the interaction between geometry and dimensional scaling of the acceptance probability. To make the computation explicit, we specifically look at a Gaussian measure on a high-dimensional sphere with radius and centre scaling like $\sqrt d$ (a particularly scaled version of the von Mises--Fisher density~\cite{wood:1994, mardia-jupp:2000}). This choice permits an asymptotic analysis, demonstrating the nontrivial influence of the geometry, while serving as benchmark for proposal mechanisms for arbitrary manifolds. For this problem, uncorrected tangent pCN has critical proposal step size scaling of $d^{-1/6}$, compared with $d^{-1/2}$ for tangent random walk, resulting in an algorithm that can sample in millions of dimensions. We prove a joint diffusion limit, on the same accelerated time scale, for parallel and orthogonal coordinates to the sphere's centre vector. The orthogonal diffusion speed coincides with the limiting normalised geodesic expected squared jump distance, with optimal acceptance probability $0.574$, whereas the parallel speed is maximised at $0.712$. Thus the different optimal tunings reflect different observables of the same chain. Finally, an additional modification of the pCN proposal precision, taking into account the manifold's geometry, cancels the terms growing with dimension, leaving a $O_p(1)$ Metropolis defect and nondegenerate limiting acceptance at fixed step size in the infinite dimensional limit of the sphere benchmark. Empirical tests sampling conditioned stochastic differential equations show similar nondegenerate limiting acceptance under temporal refinement which confirm that our geometric considerations permit dimension-invariant scaling beyond the spherical case.

\subsection{Sampling in flat space}
\label{sec:flat-space-prec}

Before introducing nonlinear constraints and sampling on a manifold, we first recall how in the case of Gaussian reference measure a non-vanishing acceptance probability can be maintained in the infinite-dimensional limit. As ambient reference measure, we take the standard Gaussian measure on $\RR^d$,
\begin{equation}
  \label{eq:reference-measure}
  \mu_0(dx) = (2\pi)^{-d/2}\exp\,\left(-\frac12\|x\|^2\right)\,dx = \varrho(x)\,dx\,,
\end{equation}
where $dx$ denotes Lebesgue measure on $\mathbb R^d$, and $\varrho(x)$ the density of $\mu_0$ with respect to it. Although this measure can be sampled directly, choosing it as target demonstrates the effect of the proposal kernel on the dimensional scaling of the Metropolis-Hastings acceptance rate. For this, one constructs a Markov process, with transition density $p(x, y)$, such that it is reversible with respect to $\mu_0$. Drawing from $\mu_0$ is then equivalent to sampling the Markov process. To construct $p(x,y)$, we can consider a Markov chain of our choice with transition density $q(x,y)$ which is used to generate a proposal $y$. This proposal is accepted with the acceptance probability $a(x,y)$. From this, $p(x,y)$ is constructed as
\begin{equation}
  p(x,y) = q(x,y) a(x,y) + \delta_x(y) \int (1-a(x,u)) q(x, du)\,.
\end{equation}
In words, with probability $a(x,y)$, we accept the proposal $y$, otherwise we reject, meaning we remain at $x$. In order to obtain detailed balance for $p(x,y)$, one needs
\begin{equation}
  a(x,y) = 1 \wedge\exp(\Rcancel(x,y))
\end{equation}
where
\begin{equation}
  \Rcancel(x,y) = \log \left(\frac{q(y,x)\varrho(y)}{q(x,y)\varrho(x)}\right) = \underbrace{\log\left(\frac{\varrho(y)}{\varrho(x)}\right)}_{\Rgauss} + \underbrace{\log\left(\frac{q(y,x)}{q(x,y)}\right)}_{\Rprop}\,,
\end{equation}
is the log acceptance ratio, which we also call the \emph{Metropolis defect}. For a Gaussian random walk, i.e.~$q(x,y) \propto \exp(-\tfrac12 \beta^{-2} \|x-y\|^2)$ with $\beta>0$, we have $\Rprop=0$ due to symmetry. This is the \emph{random walk Metropolis-Hasting} (RWMH) algorithm. Unfortunately, this comes with notoriously disadvantageous dimensional scaling, as is easily seen: If we draw a proposal $y = x + \beta v$, where $v\sim \mathcal N(0,I_d)$, then
\begin{equation}
  \Rcancel = \Rgauss = -\tfrac12(\|y\|^2-\|x\|^2) = -\beta\langle x,v\rangle - \tfrac12\beta^2 \|v\|^2\,.
\end{equation}
Since $\|v\|^2$ follows a $\chi^2$-distribution, in high dimensions we have $\E [\Rcancel] = -\beta^2 d / 2$, meaning that we need to scale $\beta = O(d^{-1/2})$ to keep acceptance rates in check~\cite{roberts-gelman-gilks:1997}.

This dimensional degradation for a Gaussian target is classically solved by preconditioned Crank-Nicolson, where $q(x,y)$ is instead constructed via $y = \sqrt{1-\beta^2} x + \beta v$ with $v\sim\mathcal N(0,I_d)$. With this choice, as is easily seen by inserting, $\Rprop$ exactly cancels $\Rgauss$ (for any $\beta$), so that $R=0$ for the Gaussian target, and every proposal is accepted independently of dimension.

\subsection{Sampling on Hypersurfaces}
\label{sec:sampling-hypersurfaces}

We now add to this a nonlinear constraint. Let $\mathbb R^d$ be equipped with the Euclidean inner product $\langle\cdot,\cdot\rangle$ and induced norm $\|\cdot\|$. We consider a smooth embedded hypersurface
\begin{equation}
  \Mc = \{z \in \mathbb R^d : F(z)=0\},
\end{equation}
where $F:\mathbb R^d \to \mathbb R$ is $C^2$ and $\nabla F(z)\neq 0$ for all $z\in \Mc$. Then $\Mc$ is a codimension-one manifold, and for each $z\in \Mc$ its tangent space is
\begin{equation}
  T_z\Mc = \{v\in \mathbb R^d : \langle v,\nabla F(z)\rangle = 0\}.
\end{equation}
We write
\begin{equation}
  \nhat_z := \frac{\nabla F(z)}{\|\nabla F(z)\|}
\end{equation}
for the unit normal, and
\begin{equation}
  P_z := I - \nhat_z\nhat_z^\top
\end{equation}
for the orthogonal projector onto $T_z\Mc$. Throughout, we reserve $x,y$ for points of flat ambient (or Euclidean) space, as in the preceding section, and $z,z'$ for points on the manifold $\Mc$.

The target of interest is the Gaussian measure constrained to $\Mc$, namely
\begin{equation}
  \pi(dz) \propto \exp\,\left(-\frac12\|z\|^2\right)\,\sigma_{\Mc}(dz),
  \label{eq:manifold_target}
\end{equation}
where $\sigma_{\Mc}$ is the induced surface measure on $\Mc$. At this point, it is important to distinguish the cases of constraining to $\Mc$ as opposed to conditioning on $\Mc$, the latter incurring an additional pre-exponential correction term $(\det \nabla F(z)^T\,\nabla F(z))^{-1/2} = \|\nabla F(z)\|^{-1}$ from the coarea formula \cite{diaconis-holmes-shahshahani:2013}.

In order to sample from a measure like this,~\citet{zappa-holmes-cerfon-goodman:2018} proceed as follows: Given the current state $z\in \Mc$, construct a proposal by first drawing a tangent increment $v\in T_z\Mc$ and then projecting back to the manifold. To map $z+v$ back to $\Mc$, one then solves for a scalar $\alpha:=\alpha(z,v) \in \RR$ such that
\begin{equation}
  z' = z + v + \alpha \nhat_z \in \Mc.
\end{equation}
In practice, $\alpha$ needs to be obtained by some root-finding algorithm, such as a Newton scheme, to solve $F(z + v + \alpha \nhat_z) = 0$ for $\alpha$. Once obtained, $z'$ is the new proposal point. The reverse tangent increment at this point is given by
\begin{equation}
  v' := P_{z'}(z-z')\in T_{z'}\Mc.
\end{equation}
In order to ensure reversibility, one further needs to perform the whole procedure reversed, by projecting $z' + v'$ onto the manifold and observing whether $z$ is recovered. If not, the proposal is rejected. Further, as remarked in~\cite{zappa-holmes-cerfon-goodman:2018}, the projection step onto the manifold does not incur an additional Jacobian ratio, as the forward and backward projection Jacobian exactly cancel.

In the original variant, the algorithm just described is equipped with a proposal corresponding to a random walk in tangent space, $v\sim \mathcal N(0,\beta^2 P_z)$, with the same dimensional deterioration as flat-space RWMH. The central question of this paper is how to construct proposals such that $\Rcancel$ behaves well in high dimension, and whether its leading terms can be made to cancel similar to pCN by matching the proposal more closely to the local geometry of the manifold. We investigate these questions mainly on the shifted sphere, where all geometric quantities can be written explicitly, and rigorous proofs can be obtained.

\subsection{Shifted sphere geometry as benchmark}
\label{sec:shift-sphere-geom}

In order to make precise statements about the scaling behaviour of the algorithm, we now consider a specifically scaled and shifted sphere as a benchmark case, which serves as an exactly solvable model for understanding the Metropolis defect $\Rcancel$. The advantage of this case is that all geometric quantities, the projection step, and the reverse tangent increment can be written in closed form. The shifted sphere serves for us as the canonical solvable case for constrained sampling, much like product targets did for unconstrained MCMC \cite{roberts-gelman-gilks:1997,roberts-rosenthal:1998, beskos-pillai-roberts-etal:2013}.

Let the sphere centre $b\in \mathbb R^d$ satisfy $ \|b\| = \delta \sqrt d$, and let $r>0$. We consider the manifold
\begin{equation}
  \Mc = S^{d-1}(b,r) := \{z\in \mathbb R^d : \|z-b\|=r\}.
\end{equation}
Equivalently,
\begin{equation}
  \Mc = \{z\in \mathbb R^d : F(z)=0\},
  \qquad
  F(z):=\|z-b\|^2-r^2.
\end{equation}
The case $b=0$ is degenerate for our purposes since then $\|z\|^2$ is constant on $\Mc$, so the constrained Gaussian target becomes the uniform surface measure. The shift $b\neq 0$ is therefore what makes the problem nontrivial.

For $z\in \Mc$, we have $\nabla F(z) = 2(z-b)$, $\nabla^2 F(z) = 2I$ and $\|\nabla F(z)\| = 2r$, so the unit normal and tangent projector of Section~\ref{sec:sampling-hypersurfaces} specialise to

\begin{equation}
\nhat:=\nhat_z=\frac{z-b}{r}, \qquad P_z = I-\nhat\nhat^\top,
\end{equation}
(and $\nhat' = \nhat_{z'}$ at a proposed point $z'$), where we drop the subscript on $\nhat$ whenever the base point is clear from context. Then $z=b+r\nhat$ and, since $P_z\nhat=0$, we also have $P_z(z)=P_z(b+r\nhat)=P_z(b)$. For the constrained Gaussian target $\pi$ of Section~\ref{sec:sampling-hypersurfaces}, the squared norm in the exponent expands as
\begin{equation}
  \|z\|^2 = \|b\|^2+r^2+2r\langle b,\nhat\rangle,
  \label{eq:sphere_norm_expansion}
\end{equation}
so the target density varies on the sphere only through the scalar $\langle b,\nhat\rangle$, the remaining terms being independent of the location $z\in\Mc$ on the sphere. In particular, and as intuitive, the Gaussian tilt favours those points whose outward normal points opposite to $b$, and thus are closer to the origin.

We can see from the expansion \eqref{eq:sphere_norm_expansion} that the target $\pi$ does not determine the ambient Gaussian measure $\mu_0$. For every $\theta>0$, the ambient Gaussian
\begin{equation}
    \mu_{\theta} = \N((1-\theta^{-1})b,\theta^{-1}I_d)
\label{eq:h_reference_measure}
\end{equation}
has the same normalised surface restriction to $\Mc$ as the standard Gaussian. Indeed, on $\Mc$, $-\frac \theta2\|z-(1-\theta^{-1})b\|^2=-\frac12\|z\|^2+\mathrm{const}$. So a proposal that is built from a reference Gaussian (like pCN) involves a choice of $\theta$. Section~\ref{sec:manifold-pcn} makes the choice of $\theta=1$ and Section~\ref{sec:geom-aware-corr} relates this choice to dimensional deterioration that is studied in Section~\ref{sec:optim-scal-shift} and identifies a form of $\theta$ for which acceptance does not degrade with dimension.

\begin{rmk}[Notation]
    We use $\myrightarrow{p}$ and $\myrightarrow{d}$ for convergence in probability and in distribution respectively. We will throughout use standard Mann-Wald notation to describe orders of magnitude of quantities. Given two sequences of positive constants $(a_n), (b_n)$, we write $a_n = o(b_n)$ when $\lim_{n \to \infty} (a_n/b_n) = 0$, and $a_n = O(b_n)$ when $\limsup_{n \to \infty} (a_n/b_n) = K < \infty$. For a sequence of random variables $\{X_n\}$ and a sequence of positive constants $\{a_n\}$, we write $X_n = o_p(a_n)$ if $X_n / a_n \myrightarrow{p} 0$ as $n \to \infty$; and $X_n = O_p(a_n)$ when $X_n / a_n$ is bounded in probability as $n \to \infty$ i.e. $\lim_{C \to \infty}\limsup_{n\to\infty}\mathbb{P}(|X_n|\ge Ca_n)=0$. In particular, for a constant $c$, $X_n = c + o_p(1)$ means that $X_n \myrightarrow{p} c$.
\end{rmk}

When considering the dimensional scaling of an algorithm, we need to be explicit about the dependence of the manifold $\Mc$ on dimension. For the shifted sphere, we prescribe the scaling $\|b\|=\delta\sqrt d$ and $r=\kappa\sqrt d$ with $\delta,\kappa=O(1)$ as $d\to\infty$. This is the natural nondegenerate regime for the ambient Gaussian reference. Indeed, if $\xi\sim \mathcal N(0,I_d)$, then $\|\xi\|=O_p(\sqrt d)$, so taking $r\asymp \sqrt d$ places the shifted sphere on the same radial scale as the Gaussian shell. It also keeps the small-step condition $\|v\|/r\ll 1$ compatible with the typical proposal size $\|v\|\sim \beta\sqrt d$.

\section{Proposed algorithm}
\label{sec:proposed-algorithm}

\subsection{Tangent preconditioned Crank-Nicolson for sampling on high-dimensional manifolds}
\label{sec:manifold-pcn}

A major contribution of this work is to introduce a family of algorithms that combine the dimensional scaling advantages of pCN and others in its family with the manifold sampling capabilities of the~\citet{zappa-holmes-cerfon-goodman:2018} algorithm. If not additionally correcting for geometric terms (as we will discuss later in Section~\ref{sec:geom-aware-corr}), the resulting algorithm does not scale independently of dimension when applied to the sphere benchmark. On the flip-side, it is simple to implement and nevertheless allows sampling in very high dimensions efficiently in practice. For $\beta \in (0,1)$, set
\begin{equation}
  \rho := \sqrt{1-\beta^2}\,, \qquad \gamma := \rho - 1\,.
  \label{eq:rho_def}
\end{equation}
We use this notation throughout the remainder of the paper.

The only change relative to the scheme of Section~\ref{sec:sampling-hypersurfaces} is the law of the tangent increment: for a point $z\in\Mc$, we propose $v\in T_z\Mc$ via
\begin{equation}
  v \mid z \sim \mathcal N(m_z,\beta^2 P_z),
  \qquad
  m_z := \gamma\,P_z(z)\,.
\end{equation}
Equivalently, $v = \gamma\,P_z(z) + \beta P_z\xi$ with $\xi \sim \mathcal N(0,I_d)$. This exactly corresponds to a pCN proposal, but with mean and displacement projected into tangent space, and we hence refer to this variant as \emph{tangent pCN}. All remaining steps are those of Section~\ref{sec:sampling-hypersurfaces}: the proposal point $z' = z + v + \alpha\nhat \in \Mc$ is obtained by solving for the projection scalar $\alpha = \alpha(z,v)$, the reverse increment is $v' = P_{z'}(z-z')$, and the reversibility check rejects unless backprojecting $z' + v'$ recovers $z$; in particular the cancellation of projection Jacobians of \citet{zappa-holmes-cerfon-goodman:2018} applies verbatim.

This procedure defines the proposal kernel $q(z,z')$, and the proposed $z'\in\Mc$ is accepted with log acceptance ratio
\begin{equation}
  \Rcancel = -\frac12(\|z'\|^2-\|z\|^2) + \frac{1}{2\beta^2} \left( \|v-m_z\|^2-\|v'-m_{z'}\|^2 \right),
  \label{eq:Rcancel_velocity_form}
\end{equation}
i.e. whenever $\log u < \Rcancel$ for $u\sim U[0,1]$.

The full cancellation of the Gaussian target known from flat-space pCN does not translate to the manifold setting. In fact, using $v-m_z = P_z(z'-\rho z)$, we obtain from \eqref{eq:Rcancel_velocity_form} that
\begin{equation}
    \Rcancel = \frac1{2\beta^2} \left( \langle \nhat', z-\rho z'\rangle^2 - \langle \nhat, z' - \rho z\rangle^2\right) = \frac1{2\beta^2} \left( r_b^2 - r_f^2 \right)
    \label{eq:Rcancel_manifold_residuals}
\end{equation}
with $r_f = \langle\nhat, z'-\rho z\rangle$ being the forward residual after Gaussian cancellation, and $r_b$ correspondingly the backward residual.

\begin{algorithm}
  \caption{tangent pCN proposal step}
  \label{alg:tangent-pCN}
  \begin{algorithmic}[1]
    \Require Current state $z \in \Mc$ \State Sample $\xi \sim \mathcal N(0,I)$ \State $m_z \gets \gamma P_z z$ \State $v \gets m_z + \beta P_z\xi$ \State $z' \gets \Call{ProjectM}{z, v, \nhat}$ \If{projection failed} \State \Return $z$ \Comment{reject} \EndIf \State $v' \gets P_{z'}(z-z')$ \State $m_{z'} \gets \gamma P_{z'} z'$ \State Sample $u\sim U[0,1]$ \If{$\log u < \Rcancel$ \textbf{and} $\|z-\Call{ProjectM}{z',v', \nhat'}\|<\text{thresh}$} \State \Return $z'$ \Comment{accept} \Else \State \Return $z$ \Comment{reject} \EndIf

    \vspace{1em}
    \Function{ProjectM}{$z$, $v$, $\nhat$} \State $\alpha \gets \textsc{RootFind}(\alpha\mapsto F(z + v + \alpha \nhat))$ \State \Return $z + v + \alpha \nhat$ \EndFunction
  \end{algorithmic}
\end{algorithm}

This whole scheme is summarised as Algorithm~\ref{alg:tangent-pCN}. Here, \textsc{RootFind}($\alpha \mapsto F(z+ v+ \alpha\nhat)$) is a placeholder for any root-finding algorithm, such as a Newton solver, yielding an $\alpha\in\RR$ with $F(z+ v+ \alpha\nhat)=0$. The function \textsc{ProjectM} implements the projection step onto the manifold, given a point $z\in\Mc$ and a tangent vector $v\in T_z\Mc$. The reversibility check is implemented as requiring that $\|z-\Call{ProjectM}{z',v', \nhat'}\|$ is below a chosen threshold, which implies that the backward step from $z'$ given tangent space increment $v'$ leads back to $z$ as desired. Whenever either projection or the reverse-projection check fails, the proposal is rejected. We use the convention that its acceptance probability is zero and $\Rcancel:=-\infty$ on that event.

Note that on the sphere geometry, many steps are significantly simpler. In particular, we have the following explicit computations:
\begin{prop}[Closed-form projection on the sphere]
\label{prop:sphere_backprojection}
Consider $\Mc = S^{d-1}(b,r)$ the sphere geometry defined in Section~\ref{sec:shift-sphere-geom}. Let $z \in \Mc$ and $v \in T_z\Mc$ satisfy $\|v\|<r$. Then:
\begin{enumerate}
\item[(i)] The projection scalar is $\alpha = -r + \sqrt{r^2 - \|v\|^2}$, and the proposed point is $z' = b + v + \sqrt{r^2 - \|v\|^2}\,\nhat$.

\item[(ii)] The new unit normal is $\nhat' = (z'-b)/r = (1 + \alpha/r)\,\nhat + v/r$.

\item[(iii)] The reverse tangent increment is $v' = -v - \alpha(\nhat + \nhat')$, and $\|v'\| = \|v\|$.
\end{enumerate}
\end{prop}

\begin{pf}
All identities follow by direct calculation from $\|z-b\| = r$, $v \perp \nhat$, and $z' = z + v + \alpha\nhat \in \Mc$. Note that we select the plus root of the projection scalar characterised by $\alpha\to0$ as $v\to0$.
\end{pf}

Given these simplifications, on the sphere the following steps simplify: Given a $v\in T_z\Mc$, the new proposal point $z'$ is available explicitly without invoking any root finder. Normal vectors can be computed explicitly at $z$ and $z'$ without the need to compute gradients. Further, the reverse proposal check is unnecessary, since the geometry is symmetric under exchange of $(z,v)\leftrightarrow (z',v')$, and the reverse proposal always agrees with the original state. Lastly, the coarea term disappears, as the (unnormalised) normal vector has the same length at all points on the sphere, so we do not need to specify in~\eqref{eq:manifold_target} whether we are constraining or conditioning the measure.

\subsection{Exact Metropolis Defect on the Shifted Sphere}
\label{sec:exact-metr-defect}

The proposals we consider below differ from tangent pCN through the drift and the noise scale of the tangent increment only. We therefore introduce the general Gaussian tangent proposal
\begin{equation}
v\mid z\sim\N\bigl(c\,P_z(b),\,\sigma^2P_z\bigr), \qquad c\in\RR,\ \sigma>0\,.
\label{eq:general_proposal}
\end{equation}
Tangent pCN is the case $c=\gamma$, $\sigma=\beta$. We now derive a closed-form expression for $\Rcancel = \Rgauss + \Rprop$ on the sphere on a successful projection, that is, on $\{\|v\|<r\}$ with the choice of $\alpha$ in Proposition~\ref{prop:sphere_backprojection}. Throughout, on $\{\|v\|\ge r\}$ we reject and retain the convention $\Rcancel=-\infty$. We first define the scalar
\begin{equation}
  g := \langle v,b\rangle + \alpha\langle \nhat,b\rangle.
  \label{eq:def_g}
\end{equation}
Since $\|z'\|^2 = \|b\|^2 + r^2 + 2r\langle \nhat',b\rangle$ and $\langle \nhat',b\rangle = \langle \nhat,b\rangle + g/r$ (by Proposition~\ref{prop:sphere_backprojection}(ii)), the target contribution is
\begin{equation}
  \Rgauss
  = \tfrac12\bigl(\|z\|^2 - \|z'\|^2\bigr)
  = -g.
  \label{eq:Rgauss_sphere}
\end{equation}
The proposal log-density on $T_z\Mc$ is $\log q_z(v) = -(2\sigma^2)^{-1}\|v - c P_z(b)\|^2 + C_{d,\sigma}$, where $C_{d,\sigma}$ is independent of $z$ and $v$. Since $\|v'\| = \|v\|$, the proposal log-ratio depends only on the inner products with $b$ and the projected norms
\begin{equation}
  \Rprop = \frac{c}{\sigma^2}\bigl(\langle v',b\rangle - \langle v,b\rangle\bigr)
    - \frac{c^2}{2\sigma^2} \bigl(\|P_{z'}(b)\|^2 - \|P_z(b)\|^2\bigr).
  \label{eq:Rprop_decomp}
\end{equation}

\begin{prop}[Exact defect on the sphere]
\label{prop:Rc_sphere}
Under the notation above,
\begin{align}
  \langle v',b\rangle - \langle v,b\rangle
  &= -\bigl(2 + \alpha/r\bigr)\,g,
  \label{eq:vb_diff} \\
  \|P_{z'}(b)\|^2 - \|P_z(b)\|^2
  &= -\frac{2\langle \nhat,b\rangle}{r}\,g - \frac{g^2}{r^2}.
  \label{eq:Pb_diff}
\end{align}
Consequently, the exact Metropolis defect of the proposal \eqref{eq:general_proposal} is
\begin{equation}
  \Rcancel =
  \left(-1-\frac{2c}{\sigma^2} -\frac{c}{\sigma^2}\,\frac{\alpha}{r}
    +\frac{c^2}{\sigma^2}\,\frac{\langle\nhat,b\rangle}{r} \right) g +
  \frac{c^2}{2\sigma^2}\,\frac{g^2}{r^2}.
  \label{eq:Rcancel_general}
\end{equation}
In particular, for tangent pCN
\begin{equation}
    \Rcancel=
    \left( \frac{1-\rho}{1+\rho} + \frac{\alpha}{(1+\rho)r} \right) g + 
    \frac{1-\rho}{1+\rho}\left(  \frac{\langle \nhat,b\rangle}{r}\,g + \frac{g^2}{2r^2} \right).
  \label{eq:Rcancel_exact}
\end{equation}
\end{prop}

\begin{pf}
Identity \eqref{eq:vb_diff} follows from Proposition~\ref{prop:sphere_backprojection}(iii) and the relation $\langle \nhat + \nhat', b\rangle = 2\langle \nhat,b\rangle + g/r$. Identity \eqref{eq:Pb_diff} follows from $\|P_z(b)\|^2 = \|b\|^2 - \langle \nhat,b\rangle^2$ and $\langle \nhat',b\rangle = \langle \nhat,b\rangle + g/r$. Inserting \eqref{eq:vb_diff}--\eqref{eq:Pb_diff} into \eqref{eq:Rprop_decomp} and using \eqref{eq:Rgauss_sphere} gives \eqref{eq:Rcancel_general}. The tangent pCN specialisation can be derived with the identities $-\gamma/\beta^2 = 1/(1+\rho)$ and $\gamma^2/\beta^2 = (1-\rho)/(1+\rho)$.
\end{pf}

Formula \eqref{eq:Rcancel_exact} makes the structure of the defect transparent. The coefficient $(1-\rho)/(1+\rho) = O(\beta^2)$ is the remnant of the flat-space drift mismatch, while $\alpha/((1+\rho)r) = O_p(\beta^2)$ is a curvature-induced correction from the sphere projection step.

\subsection{Numerical Illustration of Cancellation}
\label{sec:numer-illustr-canc}

Before turning to the asymptotic analysis, we illustrate the cancellation mechanism numerically. On the shifted sphere with $\kappa=1$, $\delta=1$, and fixed $\beta=0.3$, we compare three proposals: tangent random walk (RWMH), tangent pCN (ambient reference $\theta=1$ in \eqref{eq:h_reference_measure}) and a curvature corrected tangent pCN that we will introduce in Section~\ref{sec:geom-aware-corr}. For each proposal and dimension $d$, we run a single Metropolis--Hastings chain for $102{,}000$ proposal steps, discard the first $2\times 10^3$ iterations as burn-in, and use the remaining $10^5$ post-burn-in iterations to estimate the mean values of $\Rgauss$, $\Rprop$, and $\Rcancel$, as well as the acceptance probability.

\begin{figure}[ht]
  \centering
  \includegraphics[width=\linewidth]{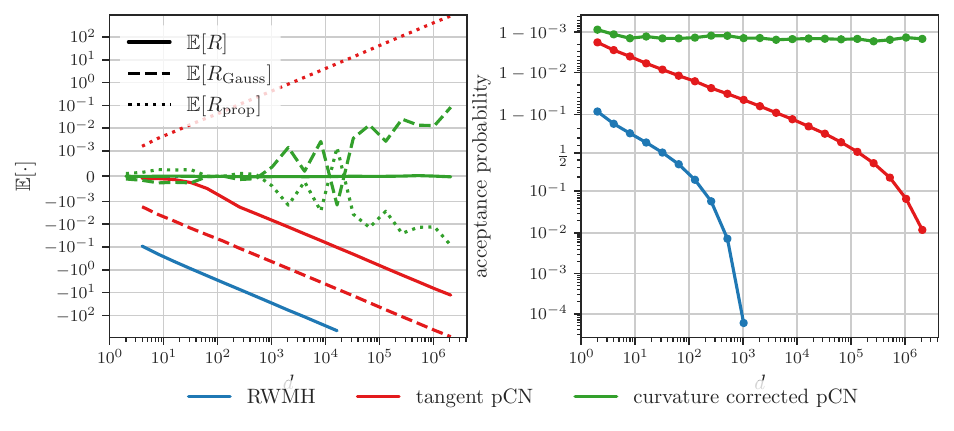}
  \caption{\small Metropolis--Hastings on a shifted sphere with radius $\sqrt d$ and centre-to-origin distance $\sqrt d$ (i.e.~$\kappa = 1, \delta = 1$), with fixed step size $\beta = 0.3$. \textbf{Left:} mean values of $R$, $\Rprop$, and $\Rgauss$ as functions of $d$. For RWMH, $\Rprop \equiv 0$ so $R = \Rgauss$. For tangent pCN, $\Rprop$ cancels most of $\Rgauss$, but a nonzero residual remains. For curvature corrected pCN, the terms in $\Rprop$ and $\Rgauss$ that grow with dimension cancel, leaving a residual independent of $d$. \textbf{Right:} acceptance probability against dimension $d$. The acceptance probability drops eventually with dimension, but tangent pCN is still usable up to a dimension of $10^5$. Curvature corrected pCN acceptance stays high for all $d$ shown.}
  \label{fig:sphere-direct-cancel}
\end{figure}

The left panel shows that $\Rgauss$ grows linearly in $d$ for both RWMH and tangent pCN, reflecting the $O(\beta^2 d)$ target contribution. For RWMH, $\Rprop \equiv 0$ (the proposal is symmetric), so $\Rcancel = \Rgauss$ and the full defect grows at the same rate. For tangent pCN, the proposal correction $\Rprop$ tracks $\Rgauss$ closely, leaving a much smaller residual $\Rcancel$. This is consistent with a partial cancellation between $\Rgauss$ and $\Rprop$; the residual $\Rcancel$ grows much more slowly than either component individually. The source of this residual is analysed in Section~\ref{sec:optim-scal-shift}. For the curvature corrected proposal, the contributions that grow with dimension cancel, leading to a $O_p(1)$ residual. Section~\ref{sec:geom-aware-corr} proves the fixed-step acceptance limit for this proposal.

The right panel of Figure~\ref{fig:sphere-direct-cancel} shows the practical consequence where, at fixed $\beta = 0.3$, the RWMH acceptance probability collapses by around $d = 10^2$, whereas pCN maintains nontrivial acceptance up to approximately $d = 10^5$ (note the unusual logit-scaling of the $y$-axis). The cancellation buys roughly three orders of magnitude in the accessible dimension range at this step size. Nevertheless, the pCN acceptance probability does eventually degrade: the geometric residual after Gaussian cancellation still grows with $d$. In contrast, we see that for curvature corrected pCN, acceptance probability stays close to $0.99$ for all $d$ tested.

While it might seem from these experiments that the curvature corrected version of tangent pCN is strictly preferable as it remains useful in the infinite-dimensional limit, we argue below that both versions need to be considered in their own right. The reason is primarily that the curvature correction, as derived in Section~\ref{sec:geom-aware-corr}, relies on the spherical geometry of our benchmark problem. While it might be possible to extend the ideas to generic shape operators of an arbitrary hypersurfaces, these will be costly to compute in a high-dimensional setting. The uncorrected variant, on the other hand, readily applies to arbitrary hypersurfaces without any further modification (even though the scaling results will not necessarily transfer). Therefore, specifically for (uncorrected) tangent pCN, the observations above raise two questions: at what rate must $\beta$ shrink in $d$ such that acceptance probability is nondegenerate and what is the optimal rate? We address both in the following section.

\section{Optimal scaling on the shifted sphere}
\label{sec:optim-scal-shift}

We now determine the scaling of the proposal parameter $\beta$ for which the acceptance probability of tangent pCN remains nondegenerate as $d\to\infty$. Throughout this section we write $r=\kappa\sqrt d$ with $\kappa>0$ and, by rotational invariance of both the ambient Gaussian measure and the shifted sphere, assume without loss of generality that $b = \delta\mathbf{1}$ with $\delta>0$. Writing $z=b+r\nhat \in S^{d-1}(b,r)$ and $\hat b:=b/\|b\| = \mathbf{1}/\sqrt{d}$, we set
\begin{equation}
  t := \langle \nhat,\hat b\rangle \in [-1,1].
  \label{eq:t_def}
\end{equation}

Using \eqref{eq:sphere_norm_expansion} together with $\|b\|=\delta\sqrt d$, $r=\kappa\sqrt d$, and $\langle\nhat,b\rangle=\delta\sqrt d\,t$, we obtain
\begin{equation}
  \|z\|^2 = (\delta^2 + \kappa^2)d + 2\delta\kappa d\,t.
\end{equation}
Therefore, the target $\pi$ \eqref{eq:manifold_target} depends on $z$ only through the scalar $t$, with one-dimensional density
\begin{equation}
  \label{eq:t_density}
  p(t) \propto (1 - t^2)^{(d-3)/2}\,e^{-\delta\kappa d\,t}.
\end{equation}
Note that when $b = 0$ (or equivalently $\delta = 0$), the tilt vanishes and $\pi$ reduces to the uniform surface measure which is a trivial target. We have that the density \eqref{eq:t_density} concentrates sharply as $d \to \infty$.

\begin{lem}[Concentration of the scalar coordinate]
\label{lem:concentration_t}
For fixed $\delta,\kappa>0$, let $t \in[-1,1]$ have density \eqref{eq:t_density}. Set
\begin{equation*}
  t^*:=\frac{1-\sqrt{1+4\delta^2\kappa^2}}{2\delta\kappa}\in(-1,0),
  \qquad
  \sigma_{t^*}^2:=\frac{(1-(t^*)^2)^2}{1+(t^*)^2}.
\end{equation*}
Then
\begin{equation}
  \label{eq:t_CLT}
  \sqrt d\,(t-t^*)\myrightarrow{d}\mathcal N(0,\sigma_{t^*}^2).
\end{equation}
Moreover, for every $k\ge1$,
\begin{equation}
  \label{eq:t_moments}
  \E_\pi\bigl[|t-t^*|^k\bigr]=O(d^{-k/2}).
\end{equation}
\end{lem}

\begin{pf}
See Appendix~\ref{app:concentration_t_proof}
\end{pf}

Throughout this section, $\pi_d$ denotes the dimension-$d$ target measure defined in Section~\ref{sec:sampling-hypersurfaces}, and all distributional statements are under the stationary law: $z\sim\pi_d$ and, conditional on $z$, the proposal increment $v\sim q_z$. We write $\mathbb E_\pi$ for expectation with respect to this joint law, including the proposal randomness.

\subsection{Reduction to scalar sufficient statistics}
\label{sec:reduct-scal-suff}

We now show that the Metropolis defect \eqref{eq:Rcancel_exact} can be reduced to a small number of scalar sufficient statistics. Conditional on the current state $z \in \Mc$, only two proposal statistics are needed. Under stationarity, the randomness of the current state contributes a third asymptotically $O(1)$ scalar. Fix $z \in \Mc$. In the tangent space $T_z\Mc$, the distinguished direction is that of the projected drift. Define $e_1 := P_z(b) / \|P_z(b)\|$ and let $e_2,\ldots,e_{d-1}$ complete an orthonormal basis of $T_z\Mc$. Then the proposal
\begin{equation*}
v=\gamma P_z(b)+\beta P_z\xi, \qquad \xi\sim\mathcal N(0,I_d),
\end{equation*}
can be written as
\begin{equation}
  \label{eq:orthonormal-basis}
  v = (\gamma\, \|P_z(b)\|+\beta\xi_1)e_1 + \beta\sum_{i=2}^{d-1}\xi_i e_i,
\end{equation}
where $\xi_1,\ldots,\xi_{d-1}\stackrel{\mathrm{iid}}{\sim}\mathcal N(0,1)$. Hence, conditional on $z$, the proposal noise enters only through the scalar $\xi_1$ and
\begin{equation}
  W:=\sum_{i=2}^{d-1}\xi_i^2\sim\chi^2_{d-2}.
  \label{eq:W_def}
\end{equation}
Indeed,
\begin{equation}
  \label{eq:scalar_reduction}
  \|v\|^2 = (\gamma \,\|P_z(b)\|+\beta\xi_1)^2 + \beta^2 W,
  \qquad
  g = (\gamma \,\|P_z(b)\|+\beta\xi_1)\|P_z(b)\| + \alpha \langle \nhat,b\rangle.
\end{equation}
Since the exact defect \eqref{eq:Rcancel_exact} expresses $\Rcancel$ in terms of $g$, $\alpha$, $\langle \nhat,b\rangle$, and $r$, it follows that conditional on the current state, $\Rcancel$ is an explicit deterministic function of the two scalar variables $(\xi_1, W)$, with parameters $\bigl(\langle \nhat,b\rangle, \|P_z(b)\|, r, \beta\bigr)$: the $(d-1)$-dimensional proposal noise enters the defect only through these two scalars.

Under stationarity, the current state contributes one further scalar degree of freedom through $t$ as defined in \eqref{eq:t_def}. Define the centred, scaled scalar coordinate by
\begin{equation}
  \zeta := \sqrt d\,(t-t^*)
    \label{eq:zeta_def}
\end{equation}
whose Gaussian limit is given by \eqref{eq:t_CLT}. We have
\begin{align}
  \langle \nhat,b\rangle
  &= \delta\sqrt d\,t
   = \delta t^*\sqrt d + \delta \zeta,
     \label{eq:nb_expansion}\\
  \|P_z(b)\|
  &= \delta\sqrt d\,\sqrt{1-t^2}
   = \tau\sqrt d + O_p(1),
     \label{eq:P_expansion}
\end{align}
where $\tau := \delta\sqrt{1-(t^*)^2}$. Let us also define the constant
\begin{equation}
s := \tau^2 + \frac{1}{\kappa^2},
\label{eq:s_def}
\end{equation}
which will appear in every proposal variant below. Moreover $\zeta$, $\xi_1$, and $W$ are mutually independent. Consequently, under stationarity, $\Rcancel$ is a deterministic function of the three scalar random variables $(\zeta,\xi_1,W)$ and the parameters $(\delta,\kappa,\beta,d)$ alone, each of the three being $O_p(1)$ after centring and scaling. The next step is to understand the size of the scalar $g=\langle v,b\rangle+\alpha\langle \nhat,b\rangle$, which appears in the exact defect \eqref{eq:Rcancel_exact}. The point is that, although both $\langle v,b\rangle$ and $\alpha\langle \nhat,b\rangle$ contain deterministic contributions of order $\beta^2 d$, these leading terms cancel under the stationary constrained Gaussian law.

\begin{lem}[Expansion of $g$]
\label{lem:g_cancellation}
Under stationarity on the shifted sphere, for $\beta=\beta_d\to0$,
\begin{equation}
  g = \beta\tau\sqrt d\,\xi_1
    + \Lambda\beta^4 d
    + O_p(\beta)
    + O_p(\beta^2\sqrt d)
    + O_p(\beta^3\sqrt d)
    + O(\beta^6 d)\,,
  \label{eq:g_final_clean}
\end{equation}
where $\Lambda := \tfrac18\tau^2 (s-1)$.
In particular, if $\beta^3\sqrt d=O(1)$, then
\begin{equation}
  g=O_p(\beta\sqrt d).
  \label{eq:g_order_statement}
\end{equation}
\end{lem}

\begin{pf}
This is Lemma~\ref{lem:g_refined}(i) in Appendix~\ref{app:Lambda_derivation} with $c_\beta=\gamma=-\tfrac12\beta^2-\tfrac18\beta^4+O(\beta^6)$, that is $\mu=-1/8$, for which $\Lambda(-1/8)=\Lambda$. If $\beta^3\sqrt d=O(1)$, then $\Lambda\beta^4 d=O(\beta\sqrt d)$ and every remainder in \eqref{eq:g_final_clean} is $O_p(\beta\sqrt d)$ which gives \eqref{eq:g_order_statement}.
\end{pf}

\subsection{Expansion at the critical scaling \texorpdfstring{$\beta=\ell d^{-1/6}$}{beta = ell d^-1/6}}
\label{sec:expans-at-crit}

We now identify the scaling of $\beta$ for which the Metropolis defect remains $O_p(1)$. By Lemma~\ref{lem:g_cancellation}, under any regime such that $\beta^3\sqrt d=O(1)$, we have $g=O_p(\beta\sqrt d)$. Splitting the exact defect \eqref{eq:Rcancel_exact} into its linear and quadratic in $g$ parts,

\begin{equation}
\Rcancel=Jg+Q, \qquad J=\frac{1-\rho}{1+\rho}+\frac{\alpha}{(1+\rho)r} +\frac{1-\rho}{1+\rho}\frac{\langle\nhat,b\rangle}{r}, \qquad
Q=\frac{1-\rho}{1+\rho}\frac{g^2}{2r^2}.
\label{eq:R_JQ_decomposition}
\end{equation}

In the small-step regime $\beta \to 0$ with $r=\kappa\sqrt d$, each of the three terms of $J$ are of order $\beta^2$: the flat-space remnant contributes $\beta^2/4$, the projection term $-\beta^2/(4\kappa^2)$ and the tilt term $-\tau^2\beta^2/4$, each up to $o_p(\beta^2)$ (Lemma~\ref{lem:coeff_general} with $c_\beta=\gamma$). Hence, $Jg=O_p(\beta^3\sqrt d)$ while $g^2/r^2=O_p(\beta^2)$ gives $Q=O_p(\beta^4)$.

A nondegenerate acceptance regime therefore requires $\beta^3\sqrt d \asymp 1$, which suggests the critical scaling $\beta=\ell\,d^{-1/6}$, with $\ell>0$ fixed. Define
\begin{equation}
C_* := \frac14\left(1 - \tau^2 - \frac{1}{\kappa^2}\right) = \frac{1-s}{4}\,,
\label{eq:Cstar_def}
\end{equation}
then summing the three contributions gives $J=C_*\beta^2+o_p(\beta^2)$, and therefore
\begin{equation}
\Rcancel = C_*\beta^2 g + o_p(1).
\label{eq:R_leading}
\end{equation}

\begin{rmk}[Exceptional parameter curve]
The definitions of $t^*$, $\tau$ and $s$ give
\begin{equation}
C_*=0\quad\Longleftrightarrow\quad s=1 \quad\Longleftrightarrow\quad \kappa^2=1+\delta^2.
\end{equation}
The $d^{-1/6}$ scaling results below assume $C_*\ne0$. On the exceptional curve, the intrinsic precision $\theta=s$ of Section~\ref{sec:geom-aware-corr} equals one, so the intrinsic proposal coincides with standard tangent pCN. Theorem~\ref{thm:fixed_step}(i) then gives a nondegenerate limiting stationary acceptance probability for every fixed $\beta\in(0,1)$.
\end{rmk}

Apart from the leading fluctuation term in \eqref{eq:g_final_clean}, the leading deterministic term $\Lambda\beta^4 d$ gives the limiting mean of $\Rcancel$. Under the critical scaling $\beta=\ell d^{-1/6}$, all other remainders in \eqref{eq:g_final_clean} are $o_p(\beta^4 d)=o_p(d^{1/3})$, so
\begin{equation}
g = \beta\tau\sqrt d\,\xi_1 + \Lambda\beta^4 d + o_p(\beta^4 d), \qquad \Lambda = -\frac{1}{2} C_*\tau^2.
\label{eq:g_critical_final}
\end{equation}

Substituting \eqref{eq:g_critical_final} into \eqref{eq:R_leading} and using $\beta=\ell d^{-1/6}$ gives
\begin{equation}
  \Rcancel = C_*\tau\,\ell^3\,\xi_1 + C_*\Lambda\,\ell^6 + o_p(1).
  \label{eq:Rcancel_limit_expansion}
\end{equation}
Note that after multiplication by the prefactor $\beta^2$, the terms $O_p(\beta)$, $O_p(\beta^2\sqrt d)$, and $O_p(\beta^3\sqrt d)$ in the expansion \eqref{eq:g_final_clean} of $g$ become $O_p(\beta^3)$, $O_p(\beta^4\sqrt d)$, and $O_p(\beta^5\sqrt d)$, respectively, all of which are $o_p(1)$ under $\beta=\ell d^{-1/6}$.

\begin{thm}[Gaussian limit]
\label{thm:gaussian_limit}
Assume $C_*\neq 0$, and let $\beta=\ell d^{-1/6}$ with $\ell>0$ fixed. Define
\begin{equation}
  \sigma_\ell := |C_*|\,\tau\,\ell^3.
\label{eq:sigmaell_def}
\end{equation}
Under the stationary proposal law,
\begin{equation}
  \Rcancel \myrightarrow{d} \mathcal N\,\left(-\frac12\sigma_\ell^2,\,\sigma_\ell^2\right).
\end{equation}
\end{thm}

\begin{pf}
Under the standing assumptions, $\tau > 0$. Hence $C_*\tau\ell^3 \neq 0$, and
\begin{equation*}
C_*\tau\,\ell^3\,\xi_1 + C_*\Lambda\,\ell^6
  \sim
  \mathcal N\,\bigl(C_*\Lambda\,\ell^6,\; C_*^2\tau^2\ell^6\bigr),
\end{equation*}
exactly for every $d$, since $\xi_1 \sim \mathcal N(0,1)$. Adding the $o_p(1)$ remainder from \eqref{eq:Rcancel_limit_expansion} gives
\begin{equation*}
  \Rcancel \myrightarrow{d} \mathcal N\,\bigl(C_*\Lambda\,\ell^6,\; C_*^2\tau^2\ell^6\bigr).
\end{equation*}
Using $\Lambda = - C_*\tau^2 /2$ and $\sigma_\ell^2 = C_*^2\tau^2\ell^6$, the mean simplifies to $C_*\Lambda\ell^6 = - \sigma_\ell^2 / 2$, giving the stated form.
\end{pf}

\begin{cor}[Limiting acceptance probability]
\label{cor:acceptance_limit}
Let
\begin{equation}
  a_d(\beta) := \E_\pi\,\left[1 \wedge e^{\Rcancel}\right]
\end{equation}
denote the stationary mean acceptance probability at step size $\beta$. Under the scaling $\beta=\ell d^{-1/6}$,
\begin{equation}
  a_d(\ell d^{-1/6}) \longrightarrow a(\ell)
  := 2\,\Phi\,\left(-\frac{\sigma_\ell}{2}\right),
\end{equation}
where $\sigma_\ell$ is given by \eqref{eq:sigmaell_def}, and $\Phi$ denotes the standard normal cumulative distribution function.
\end{cor}

\begin{pf}
By Theorem~\ref{thm:gaussian_limit},
\begin{equation}
  \Rcancel \myrightarrow{d} X \sim \mathcal N\,\left(-\frac12\sigma_\ell^2,\sigma_\ell^2\right).
\end{equation}
Since the function $x\mapsto 1 \wedge e^x $ is bounded and continuous, it follows that
\begin{equation}
  a_d(\ell d^{-1/6})=\E_\pi[1 \wedge e^{\Rcancel}] \longrightarrow \E[1 \wedge e^{X}].
\end{equation}
For $X\sim \mathcal N(-\sigma_\ell^2/2,\sigma_\ell^2)$, the latter expectation is just $2\Phi\left(-\sigma_\ell / 2\right)$. This gives the result.
\end{pf}

\subsection{Consequences for scaling and optimal tuning}
\label{sec:cons-scal-optim}

The expansion above yields the following trichotomy for the step size parameter $\beta$.

\begin{cor}[Scaling trichotomy]
\label{cor:scaling_trichotomy}
Assume $C_*\neq 0$ and let $\beta_d\to 0$. Under the stationary proposal law, define $\lambda_d := \beta_d^3\sqrt d$, then the following three regimes hold.

\begin{enumerate}
  \item If $\lambda_d\to 0$, then
  \begin{equation}
    \Rcancel \myrightarrow{p} 0,
  \end{equation}
  and hence the stationary acceptance probability satisfies
  \begin{equation}
    a_d(\beta_d)\to 1.
  \end{equation}

  \item If $\lambda_d\to \ell^3\in(0,\infty)$, then
  \begin{equation}
    \Rcancel \myrightarrow{d} \mathcal N\left(-\frac12\sigma_\ell^2,\sigma_\ell^2\right),
  \end{equation}
where $\sigma_\ell$ is defined in \eqref{eq:sigmaell_def}, and therefore
  \begin{equation}
    a_d(\beta_d)\to a(\ell)\,,
  \end{equation}
with $a(\ell)$ defined in Corollary~\ref{cor:acceptance_limit}.
  \item If $\lambda_d\to\infty$, then
  \begin{equation}
    \Rcancel \myrightarrow{p} -\infty,
  \end{equation}
  and hence
  \begin{equation}
    a_d(\beta_d)\to 0.
  \end{equation}
\end{enumerate}
\end{cor}

\begin{pf}
Proof given in Appendix~\ref{app:scaling_trichotomy_proof}
\end{pf}

The critical regime therefore has exponent $d^{-1/6}$. To choose the tuning constant $\ell$, we can optimise the expected squared jump distance (ESJD) as is standard practice in optimal scaling~\cite{roberts-rosenthal:1998,beskos-pillai-roberts-etal:2013}. On the sphere, the natural measure of progress is the squared geodesic distance~\cite{lie-rudolf-sprungk-etal:2023}, so we define
\begin{equation}
\mathrm{ESJD} := \E_\pi\,\bigl[ d_{\mathrm{geo}}(z,z')^2\,a(z,v)\bigr],
  \label{eq:ESJD_def}
\end{equation}
where $d_{\mathrm{geo}}$ is the geodesic distance on $S^{d-1}(b,r)$ and $a(z,v)$ is the acceptance probability. The main result of this subsection is the following.

\begin{thm}[Limiting ESJD]
\label{thm:ESJD_limit}
Assume $C_* \neq 0$ and $\beta = \ell d^{-1/6}$ with $\ell > 0$ fixed. Under the stationary proposal law,
\begin{equation}
  \lim_{d \to \infty}
  d^{-2/3}\,\mathrm{ESJD}_d(\ell)
  = \mathcal E(\ell)
  := \ell^2 \,2\Phi\,\left(-\tfrac12 \sigma_\ell\right),
\label{eq:ESJD_limit}
\end{equation}
where $\sigma_\ell$ is defined in \eqref{eq:sigmaell_def}.
\end{thm}

We first state the following lemmas to aid us in a proof for the above theorem.

\begin{lem}[Geodesic length of a projected step]
\label{lem:geodesic_identity}
For any $z \in \Mc$ and $v \in T_z\Mc$ with $\|v\| < r$, let $z' = z + v + \alpha\nhat \in \Mc$ be the projected proposal from Proposition~\ref{prop:sphere_backprojection}. Then
\begin{equation}
  d_{\mathrm{geo}}(z,z') = r\arcsin\,\bigl(\|v\|/r\bigr).
  \label{eq:geodesic_exact}
\end{equation}
In particular,
\begin{equation}
  \bigl|\,d_{\mathrm{geo}}(z,z')^2 - \|v\|^2\,\bigr| \le C\,\|v\|^4/r^2
  \label{eq:geodesic_expand}
\end{equation}
for a constant $C>0$ independent of $z$, $v$ and $r$.
\end{lem}

\begin{pf}
By definition $d_{\mathrm{geo}}(z,z') = r \arccos(\langle\nhat,\nhat'\rangle)$. Proposition~\ref{prop:sphere_backprojection}(ii) gives $d_{\mathrm{geo}}(z,z')= r\arccos\,\bigl(\sqrt{1 - \|v\|^2/r^2}\bigr) = r\arcsin\,\bigl(\|v\|/r\bigr)$, using $\arccos(\sqrt{1-u^2}) = \arcsin(u)$ for $u \in [0,1]$. This proves \eqref{eq:geodesic_exact}. The function $u\longmapsto (\arcsin^2u-u^2) /u^4$ is nonnegative and extends continuously to $[0,1]$, with a value of $1/3$ at $u=0$. Its boundedness gives the claimed bound \eqref{eq:geodesic_expand}.
\end{pf}

\begin{lem}[Moments of the tangent increment]
\label{lem:v2_moment}
For any $z \in \Mc$, under the tangent pCN proposal,
\begin{equation}
  \E[\|v\|^2 \mid z] = \beta^2(d-1) + \gamma^2 \|P_z(b)\|^2.
  \label{eq:v2_mean}
\end{equation}
Define $M_d := \E_\pi[\|v\|^2]$, then under the critical scaling $\beta = \ell d^{-1/6}$ and stationarity,
\begin{equation}
  M_d = \beta^2(d-1)\bigl(1 + O(\beta^2)\bigr)
  = \ell^2 d^{2/3}\bigl(1 + o(1)\bigr),
  \label{eq:Md_leading}
\end{equation}
\begin{equation}
  \Var_\pi(\|v\|^2) = 2\beta^4(d-1)\bigl(1+o(1)\bigr),
  \label{eq:v2_total_var}
\end{equation}
and consequently
\begin{equation}
  \E_\pi\,\left[\left(\frac{\|v\|^2}{M_d} - 1\right)^2\right] = O(d^{-1}).
  \label{eq:eta_L2}
\end{equation}
\end{lem}

\begin{pf}
Conditional on $z$, applying the Gaussian second-moment identity to the tangent pCN proposal gives $\E[\|v\|^2 \mid z] = \| \gamma P_z(b) \|^2 + \mathrm{tr}(\beta^2 P_z)$ with $\mathrm{tr}(P_z) = d-1$ gives \eqref{eq:v2_mean}. Averaging under $\pi$, we have $M_d = \beta^2(d-1) + \gamma^2 \E_\pi[\|P_z(b)\|^2]$. We can write $\|P_z(b)\|^2=\delta^2 d(1-t^2)$. Since $t \myrightarrow{p} t^*$ by Lemma~\ref{lem:concentration_t} and $0 \le1 - t^2 \le 1$, bounded convergence gives $\E_\pi[1-t^2] \to 1-(t^*)^2$. Therefore, $\E_\pi[\|P_z(b)\|^2] = \delta^2 d\,\E_\pi[1-t^2] = \delta^2 d(1-(t^*)^2)(1+o(1)) = \tau^2 d(1+o(1))$. Since $\gamma^2=\beta^4 / 4+O(\beta^6)$, it follows that $\gamma^2\E_\pi[\|P_z(b)\|^2]=O(\beta^4 d)$, and hence $M_d=\beta^2(d-1)\bigl(1+O(\beta^2)\bigr)$. Under $\beta=\ell d^{-1/6}$, this becomes $M_d=\ell^2 d^{2/3}(1+o(1))$, proving \eqref{eq:Md_leading}.

For the total variance, basis decomposition gives $\Var(\|v\|^2 \mid z) = 2\beta^4(d-1) + 4\beta^2\gamma^2\|P_z(b)\|^2$, using $\Var(X^2)=2\sigma^4+4\sigma^2\mu^2$ for $X\sim\mathcal N(\mu,\sigma^2)$. We will use $\E_\pi[\|P_z(b)\|^2] = O(d)$ and $\Var_\pi(\|P_z(b)\|^2) = O(d)$, both of which follow from $\|P_z(b)\|^2 = \delta^2 d(1-t^2)$ together with Lemma~\ref{lem:concentration_t}. By the law of total variance, the first term contributes
\begin{equation*}
  \E_\pi[\Var(\|v\|^2 \mid z)]
  = 2\beta^4(d-1) + 4\beta^2\gamma^2\E_\pi[\|P_z(b)\|^2]
  = 2\beta^4(d-1)(1+o(1)),
\end{equation*}
using $\gamma^2 \E_\pi[\|P_z(b)\|^2] = O(\beta^4 d)$, and the second contributes
\begin{equation*}
  \Var_\pi(\E[\|v\|^2 \mid z]) = \gamma^4 \Var_\pi(\|P_z(b)\|^2)
  = O(\beta^8 d) = o(\beta^4 d).
\end{equation*}
Hence $\Var_\pi(\|v\|^2) = 2\beta^4(d-1)(1+o(1))$, proving \eqref{eq:v2_total_var}. Finally, since $\|v\|^2/M_d - 1$ has mean zero,
\begin{equation*}
  \E_\pi\,\left[\left(\frac{\|v\|^2}{M_d} - 1\right)^2\right]
  = \frac{\Var_\pi(\|v\|^2)}{M_d^2}
  = \frac{O(\beta^4 d)}{O(\beta^4 d^2)}
  = O(d^{-1}),
\end{equation*}
proving \eqref{eq:eta_L2}.
\end{pf}

\begin{lem}[Asymptotic factorisation]
\label{lem:factorisation}
Under the critical scaling $\beta = \ell d^{-1/6}$,
\begin{equation}
  \E_\pi[\|v\|^2 \cdot a(z,v)]
  = M_d \cdot \E_\pi[a(z,v)] + o(M_d)\,,
  \label{eq:factorisation}
\end{equation}
where $a(z,v)$ is the acceptance probability of the scheme.
\end{lem}

\begin{pf}
Write $\|v\|^2 = M_d(1 + \eta_d)$ where $\eta_d := \|v\|^2/M_d - 1$. Then
\[
  \E_\pi[\|v\|^2 \cdot a]
  = M_d \E_\pi[a] + M_d \E_\pi[\eta_d \cdot a].
\]
Since $0 \le a \le 1$, Cauchy--Schwarz gives
\[
  |\E_\pi[\eta_d a]|
  \le \E_\pi[\eta_d^2]^{1/2} \E_\pi[a^2]^{1/2}
  \le \E_\pi[\eta_d^2]^{1/2}
  = O(d^{-1/2}),
\]
using \eqref{eq:eta_L2}. Hence $M_d \E_\pi[\eta_d a] = o(M_d)$.
\end{pf}

With Lemmas~\ref{lem:geodesic_identity}--\ref{lem:factorisation} in hand, we are now ready to prove Theorem~\ref{thm:ESJD_limit}.

\begin{pf}
A proposal with $\|v\|\geq r$ is rejected by the convention of Section~\ref{sec:exact-metr-defect}, so $a(z,v)=0$ on $\{\|v\| \geq r\}$ and the integrand in \eqref{eq:ESJD_def} vanishes there; all expectations below thus reduce to integrals over $\{\|v\| <  r\}$, where Lemma~\ref{lem:geodesic_identity} applies. Splitting $d_{\mathrm{geo}}^2=\|v\|^2+(d_{\mathrm{geo}}^2-\|v\|^2)$ and using that $\|v\|^2 a$ also vanishes on $\{\|v\| \ge r\}$,
\begin{equation}
  \mathrm{ESJD}_d(\ell)
  = \E_\pi\bigl[\|v\|^2\,a(z,v)\bigr]
  + \E_\pi\bigl[\bigl(d_{\mathrm{geo}}(z,z')^2-\|v\|^2\bigr)\,a(z,v)\bigr].
  \label{eq:ESJD_split}
\end{equation}

Since $a\le\mathds 1_{\{\|v\| < r\}}$, the bound \eqref{eq:geodesic_expand} gives
\begin{equation*}
  \bigl|\E_\pi[(d_{\mathrm{geo}}^2-\|v\|^2)\,a]\bigr|
  \le \E_\pi\bigl[|d_{\mathrm{geo}}^2-\|v\|^2|\,\mathds 1_{\{\|v\| < r\}}\bigr]
  \le \frac{C}{r^2}\,\E_\pi[\|v\|^4].
\end{equation*}
Using $\eta_d$ defined in the proof of Lemma~\ref{lem:factorisation}, the mean-zero property of $\eta_d$ and \eqref{eq:eta_L2} give $\E_\pi[\|v\|^4]=M_d^2(1+\E_\pi\eta_d^2)=M_d^2(1+O(d^{-1}))$. With $r=\kappa\sqrt d$ and $M_d=\ell^2 d^{2/3}(1+o(1))$ from \eqref{eq:Md_leading},
\begin{equation}
  \bigl|\E_\pi[(d_{\mathrm{geo}}^2-\|v\|^2)\,a]\bigr|
  \le \frac{C\,M_d^2(1+o(1))}{\kappa^2 d}
  = O(d^{1/3}).
  \label{eq:geo_correction_bound}
\end{equation}

By Lemma~\ref{lem:factorisation}, $\E_\pi[\|v\|^2 a]=M_d\,\E_\pi[a]+o(M_d)$. Substituting this and \eqref{eq:geo_correction_bound} into \eqref{eq:ESJD_split},
\begin{equation*}
  \mathrm{ESJD}_d(\ell)=M_d\,\E_\pi[a]+o(M_d)+O(d^{1/3}).
\end{equation*}
Dividing by $d^{2/3}$, and using $M_d/d^{2/3}=\ell^2(1+o(1))$ from \eqref{eq:Md_leading}, with $o(M_d)/d^{2/3}=o(1)$ and $O(d^{1/3})/d^{2/3}=o(1)$, we obtain
\begin{equation*}
  d^{-2/3}\,\mathrm{ESJD}_d(\ell)=\ell^2(1+o(1))\,\E_\pi[a]+o(1).
\end{equation*}
Finally $\E_\pi[a]=a_d(\ell d^{-1/6})\to 2\Phi(-\tfrac12\sigma_\ell)$ by Corollary~\ref{cor:acceptance_limit}, so the right-hand side converges to $\ell^2\cdot 2\Phi(-\tfrac12\sigma_\ell)=\mathcal E(\ell)$, which is \eqref{eq:ESJD_limit}.
\end{pf}

Maximising the limiting efficiency $\mathcal E(\ell)$ over $\ell$ now yields the optimal tuning. Setting
\begin{equation}
  x := \frac{\sigma_\ell}{2} = \frac{|C_*|\,\tau}{2}\,\ell^3,
  \label{eq:x_def}
\end{equation}
so that $\ell^2\propto x^{2/3}$, maximising $\mathcal E(\ell)=\ell^2\cdot2\Phi(-\sigma_\ell/2)$ over $\ell>0$ is equivalent to maximising $x^{2/3}\,\Phi(-x)$ over $x>0$. The first-order condition is
\begin{equation}
  \frac{2}{3}\,\Phi(-x) = x\,\phi(x)\,,
  \label{eq:xstar_eq}
\end{equation}
where $\phi$ denotes the standard normal density. Since $x\mapsto x\,\phi(x)/\Phi(-x)$ is strictly increasing on $(0,\infty)$ and grows from $0$ to $\infty$, equation \eqref{eq:xstar_eq} has a unique positive root, $x_* \approx 0.561824$. Hence the asymptotically ESJD-optimal acceptance probability and tuning constant of tangent pCN are
\begin{equation}
  a_* := 2\Phi(-x_*) \approx 0.574,
  \qquad
  \ell_* := \left(\frac{2x_*}{|C_*|\,\tau}\right)^{1/3};
  \label{eq:astar_ellstar}
\end{equation}
for $\delta=\kappa=1$, $\ell_*\approx 2.10$.

\subsection{Numerical verification of the optimal scaling}
\label{sec:numer-verif-optim}

We verify the predictions of Theorem~\ref{thm:gaussian_limit} and Corollary~\ref{cor:acceptance_limit} by exact computation on the shifted sphere. Since all relevant quantities are available in closed form, we bypass MCMC entirely: for each dimension~$d$, we draw $10^6$ independent samples of the scalar stationary statistic $t$ from its exact density \eqref{eq:t_density}, together with independent proposal statistics $ \xi_1 \sim \mathcal N(0,1)$, $ W \sim \chi^2_{d-2}$, and evaluate the proposal, projection, acceptance probability, jump size, and $\Rcancel$ exactly from the closed-form formulas. We fix $\delta=\kappa=1$ throughout, giving $t^* \approx -0.618$, $\tau \approx 0.786$, $C_* \approx -0.155$, and $\ell_* \approx 2.10$.

\begin{figure}[ht!]
  \centering
  \includegraphics[width=\linewidth]{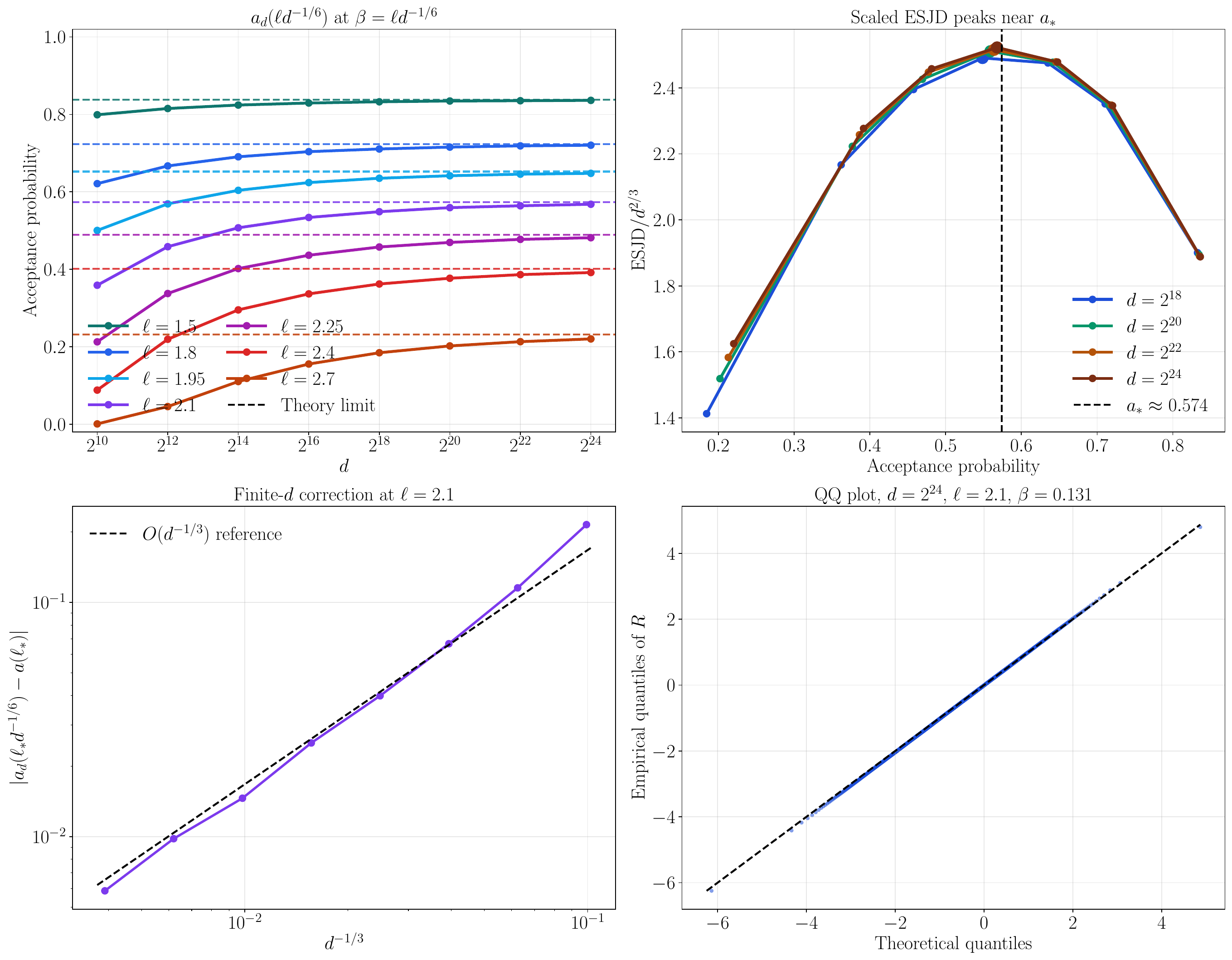}
  \caption{\small Numerical verification of the optimal scaling theory on the
    shifted sphere with $\delta=\kappa=1$, over dimensions $d \in \{2^{10},2^{12},2^{14},2^{16},2^{18},2^{20},2^{22},2^{24}\}$. \textbf{(a)} Acceptance probability $a_d(\ell d^{-1/6})$ versus $d$ for $\beta=\ell d^{-1/6}$ and several fixed values of~$\ell$. The coloured solid curves are the empirical acceptances; the matching dashed horizontal lines are the theoretical limits $a(\ell)$ defined in Corollary~\ref{cor:acceptance_limit}. \textbf{(b)} Normalised geodesic expected squared jump distance $\mathrm{ESJD}/d^{2/3}$ versus acceptance probability. The peak is attained near $\ell_* \approx 2.10$, corresponding to the ESJD-optimal acceptance probability $a_* \approx 0.574$. \textbf{(c)} Finite-dimensional correction: $|a_d(\ell_*d^{-1/6})-a(\ell_*)|$ versus $d^{-1/3}$ at $\ell=\ell_*$. The dashed line has unit slope, which shows empirically a leading $O(d^{-1/3})$ correction. \textbf{(d)} QQ plot of $\Rcancel$ at $d=2^{24}$ and $\ell=\ell_*$ against the theoretical $\mathcal N(-\sigma_\ell^2/2,\sigma_\ell^2)$ distribution from Theorem~\ref{thm:gaussian_limit}.}
  \label{fig:optimal_scaling}
\end{figure}
Figure~\ref{fig:optimal_scaling}(a) shows that the empirical acceptance probability converges to the limiting value $a(\ell)$ of Corollary~\ref{cor:acceptance_limit} as $d$ increases. At $d=2^{24}$, the agreement is within about $1\%$ at $\ell=2.1$; convergence is slower for larger $\ell$. The residual gap at moderate dimensions is consistent with a leading finite-dimensional correction of order $d^{-1/3}$. Panel~(b) shows that the normalised efficiency $\mathrm{ESJD}/d^{2/3}$ is maximised near $\ell_* \approx 2.10$, with the corresponding acceptance probability close to the theoretical optimum $a_* \approx 0.574$. Panel~(c) plots the finite-dimensional correction $|a_d(\ell_*d^{-1/6})-a(\ell_*)|$ against $d^{-1/3}$. The data follows the reference line consistently with the $O(d^{-1/3})$ correction. Finally, panel~(d) provides direct evidence for the Gaussian limit in Theorem~\ref{thm:gaussian_limit}; at $d=2^{24}$ and $\ell=2.1$, the empirical quantiles of $\Rcancel$ closely track those of the predicted $\mathcal N(-\sigma_\ell^2/2,\sigma_\ell^2)$ law.

\subsection{A drift family of proposals}
\label{sec:drift-family}

The ESJD-optimal acceptance probability $a_* \approx 0.574$ of tangent pCN coincides with that of MALA as presented in~\cite{roberts-rosenthal:1998}. This is not a coincidence, indeed both proposals are two nearby members of the same family. Tangent pCN and MALA are both members of a natural one-parameter family of drift-corrected tangent proposals, \eqref{eq:general_proposal} with $\sigma=\beta$,
\begin{equation}
  v \mid z \;\sim\; \N\bigl(c_\beta\,P_z(b),\ \beta^2 P_z\bigr),
  \qquad c_\beta \in \RR\,.
  \label{eq:drift_family}
\end{equation}
The choice $c_\beta \equiv 0$ is the tangent random-walk proposal (tangent RWMH); the choice $c_\beta = -\beta^2/2$ is tangent MALA, i.e.\ the Euler--Maruyama discretisation of the Langevin diffusion in the tangent step,
\begin{equation*}
v = \frac{\beta^2}{2}\,\mathrm{grad}_{\Mc}\log \frac{d\pi}{d\sigma_{\Mc}}(z) + \beta P_z\xi
 = -\frac{\beta^2}{2}P_z(b) + \beta P_z\xi\,,
\end{equation*}

followed by the backprojection of Proposition~\ref{prop:sphere_backprojection}, in the convention of the flat-space algorithm~\cite{roberts-rosenthal:1998,girolami-calderhead:2011}. Lastly, $c_\beta = \gamma = \sqrt{1-\beta^2}-1$ corresponds to tangent pCN. Every member defines a valid Metropolis--Hastings algorithm in the scheme of Section~\ref{sec:sampling-hypersurfaces}, since the acceptance ratio accounts for the drift. Its exact defect is \eqref{eq:Rcancel_general} with $c=c_\beta$.

\begin{prop}[Critical scaling for the drift family]
\label{prop:drift_family_scaling}
Let $c_\beta = -\tfrac12\beta^2 + \mu\beta^4 + o(\beta^4)$ with $\mu\in\RR$ fixed, assume stationarity, and let $\beta=\ell d^{-1/6}$ with $\ell>0$ fixed. Then:
\begin{enumerate}
\item[(i)] By Lemma~\ref{lem:g_refined}(i), the expansion \eqref{eq:g_critical_final} holds with $\Lambda$ replaced by
\begin{equation}
\Lambda(\mu) := \tau^2\Bigl(\mu + \frac{s}{8}\Bigr),
\label{eq:Lambda_mu_def}
\end{equation}
the drift contributing $\mu\tau^2\beta^4 d$ through $c_\beta\|P_z(b)\|^2$ while the two curvature contributions are as for tangent pCN.
\item[(ii)] By Lemma~\ref{lem:coeff_general}, the coefficient of $g$ in \eqref{eq:Rcancel_general} with $c=c_\beta$ and $\sigma=\beta$ is $C_*(\mu)\beta^2+o_p(\beta^2)$, whence
\begin{equation}
  \Rcancel = C_*(\mu)\,\beta^2 g + o_p(1),
  \qquad
  C_*(\mu) := -2\mu - \frac{s}{4},
  \qquad\text{and}\qquad
  \Lambda(\mu) = -\tfrac12 C_*(\mu)\,\tau^2 .
  \label{eq:Cstar_mu}
\end{equation}
\item[(iii)] Consequently, if $C_*(\mu)\neq 0$, the conclusions of Theorem~\ref{thm:gaussian_limit}, Corollary~\ref{cor:acceptance_limit} and Theorem~\ref{thm:ESJD_limit} hold with
\begin{equation}
  \sigma_\ell(\mu) = |C_*(\mu)|\,\tau\,\ell^3,
  \label{eq:sigmaell_mu_def}
\end{equation}
so the limiting ESJD-optimal acceptance probability is $2\Phi(-x_*)\approx 0.574$ for every member of the family, attained at $\ell_*(\mu) = \bigl(2x_*/(|C_*(\mu)|\tau)\bigr)^{1/3}$.
\end{enumerate}
\end{prop}

\begin{pf}
The expansion of $g$ follows from Lemma~\ref{lem:g_refined}(i), while the coefficient expansion in the defect follows from Lemma~\ref{lem:coeff_general}. Together these give the identities in \eqref{eq:Cstar_mu}. The remainder of the argument is verbatim Sections~\ref{sec:expans-at-crit}--\ref{sec:cons-scal-optim}.
\end{pf}

\begin{cor}[Tangent MALA]
\label{cor:tangent_mala}
Tangent MALA is the member $\mu=0$, with
\begin{equation}
  C_*^{\mathrm{MALA}} = -\frac{s}{4} < 0
  \qquad\text{for all }\delta,\kappa>0 .
\end{equation}
Hence tangent MALA obeys the critical scaling $\beta=\ell d^{-1/6}$ with $\sigma_\ell(0)=(s/4)\,\tau\ell^3$ by \eqref{eq:sigmaell_mu_def}, and has limiting ESJD-optimal acceptance $\approx 0.574$, and the exceptional cancellation $C_*^{\mathrm{MALA}}=0$ never occurs for tangent MALA.
\end{cor}

Note the structural difference from tangent pCN: for MALA the flat-space coefficient $-1-2c_\beta/\beta^2$ vanishes exactly (the analogue of the remnant $(1-\rho)/(1+\rho)$ in \eqref{eq:Rcancel_exact} is identically zero), so $C_*^{\mathrm{MALA}}$ consists purely of the two curvature terms. For pCN, $\mu=-1/8$ gives $C_* = (1-s)/4$, recovering \eqref{eq:Cstar_def}: the flat-space remnant $+\tfrac14$ partially offsets the curvature contribution $-s/4$.

\begin{cor}[Tangent RWMH]
\label{cor:tangent_rw}
For $c_\beta\equiv 0$ (RWMH) the proposal log-ratio vanishes identically, so $\Rcancel=-g$ exactly. By Lemma~\ref{lem:g_refined}(ii), the $O(\beta^2 d)$ cancellation present for tangent pCN does not occur, and instead $g = \beta\tau\sqrt d\,\xi_1 + \tfrac12\tau^2\beta^2 d\,(1+o_p(1))$. Hence with $\beta=\ell d^{-1/2}$, $\ell>0$ fixed,
\begin{equation}
  \Rcancel \myrightarrow{d} \N\Bigl(-\tfrac12\tau^2\ell^2,\ \tau^2\ell^2\Bigr),
  \qquad
  a_d(\ell d^{-1/2})\to 2\Phi\Bigl(-\frac{\tau\ell}{2}\Bigr),
  \qquad
  \mathrm{ESJD}_d(\ell)\to \ell^2\cdot 2\Phi\Bigl(-\frac{\tau\ell}{2}\Bigr),
\end{equation}
with ESJD-optimal acceptance $2\Phi(-x_*^{\mathrm{RW}})\approx 0.234$, where $x_*^{\mathrm{RW}}\approx 1.1906$ solves $2\Phi(-x)=x\phi(x)$, attained at $\ell_* = 2x_*^{\mathrm{RW}}/\tau$.
\end{cor}

By \eqref{eq:Cstar_mu}, $|C_*^{\mathrm{pCN}}| < |C_*^{\mathrm{MALA}}|$ iff $|1-s|<s$ iff $s>\tfrac12$, which is equivalent to $\kappa^2 < 4\delta^2+2$. Therefore, tangent pCN dominates tangent MALA precisely in a curvature dominated regime whereas tangent MALA dominates in a near-flat regime $\kappa^2 > 4\delta ^2 + 2$; the asymptotic ESJD ratio at the respective optimal tunings is $(s/|1-s|)^{2/3}$. At $\delta=\kappa=1$ one has $C_*^{\mathrm{MALA}}\approx-0.4045$ versus $C_*^{\mathrm{pCN}}\approx-0.1545$ and optimal tunings $\ell_*^{\mathrm{MALA}}\approx 1.52$ versus $\ell_*^{\mathrm{pCN}}\approx 2.10$.

\subsection{A joint diffusion limit}
\label{sec:diffusion-limit}

The ESJD criterion of Theorem~\ref{thm:ESJD_limit} is a useful surrogate for efficiency as it measures the total squared displacement. To relate this criterion to the dynamics of individual observables, we prove a joint diffusion limit for the longitudinal fluctuation and transverse coordinate for tangent pCN. We show here that both evolve on the $d^{1/3}$ time scale, but their limiting speeds have different optimal tunings.

For a metric space $E$, let $\mathcal B(E)$ denote its Borel $\sigma$-algebra and $\mathcal P(E)$ the set of Borel probability measures on $E$. We write $D([0,\infty),\RR^n)$ for the space of c\`adl\`ag paths equipped with the Skorokhod $J_1$ topology \citep{ethier-kurtz:2009}. Weak convergence in this space is denoted by $\Rightarrow$.

Let $(z_k)_{k\ge0}$ be the tangent-pCN chain and recall $t(z)$ from \eqref{eq:t_def}. For each $d$, let us choose a deterministic unit vector $e_{\perp} \perp b$ and set
\begin{equation}
    X_k:=\langle z_k-b,e_{\perp}\rangle.
\label{eq:Xk_def}
\end{equation}

We will call $t$ and its centred fluctuation $\zeta$ \emph{longitudinal} since $t$ is the normalised coordinate along $b$ and we will call $X$ \emph{transverse} since $e_{\perp}$ is orthogonal to $b$. Of course, the choice of $e_{\perp}$ is arbitrary by rotational symmetry about $b$.

\begin{lem}[Markovianity]
\label{lem:scalar_markovianity}
The sequences $t_k:=t(z_k)$ and $(t_k,X_k)$ are time-homogeneous Markov chains. Their transition kernels preserve the corresponding images of $\pi_d$; in particular, the scalar kernel preserves \eqref{eq:t_density}. The same statements hold for the intrinsic proposal \eqref{eq:intrinsic_proposal}.
\end{lem}

\begin{pf}
The scalar proposal and acceptance probability are measurable functions of $(t(z),\xi_1,W)$ and the law of $(\xi_1,W)$ does not depend on $z$. At $t=\pm1$, the choice of the tangent basis is irrelevant. For the pair, orthogonal transformations fixing both $b$ and $e_{\perp}$ act transitively on each set of states with a given $(t,X)$, and the proposal, projection and acceptance are equivariant under these transformations. Thus each projected transition law depends only on the corresponding projected state. Preservation of the image laws follows from invariance of $\pi_d$, its image under $z \mapsto t(z)$ preserves the density \eqref{eq:t_density}. The same argument applies after the constant rescaling of drift and noise in \eqref{eq:intrinsic_proposal}.
\end{pf}

Recall $\zeta=\sqrt d(t-t^*)$ from \eqref{eq:zeta_def} and note that, on a successful projection, Proposition~\ref{prop:sphere_backprojection}(ii) gives $t(z')-t(z)=g/(\delta\kappa d)$. If $\mathds 1_{\mathrm{acc}}$ is the acceptance indicator, we have
\begin{equation}
 \Delta\zeta:=\zeta_{k+1}-\zeta_k
 =\mathds{1}_{\mathrm{acc}}\frac{g}{\delta\kappa\sqrt d}.
 \label{eq:delta_zeta}
\end{equation}
Under the critical scaling $\beta=\ell d^{-1/6}$, \eqref{eq:g_critical_final} gives $g=O_p(d^{1/3})$ and hence $\Delta\zeta=O_p(d^{-1/6})$. A one-step variance is therefore of order $d^{-1/3}$, which identifies the accelerated process
\begin{equation}
 S^{(d)}(u):=\zeta_{\lfloor d^{1/3}u\rfloor},
 \qquad X^{(d)}(u) :=X_{\lfloor d^{1/3}u\rfloor},
 \qquad u\ge0\,,
 \label{eq:Sd_def}
\end{equation}
i.e. the longitudinal coordinate centred at $t^*$ and magnified by $\sqrt d$, and the transverse coordinate, both run at $d^{1/3}$ steps per unit time.

\begin{thm}[Joint diffusion limit]
\label{thm:joint_diffusion_limit}
Fix $\delta,\kappa,\ell>0$, assume $C_*\ne0$, take $\beta=\ell d^{-1/6}$, and start the full chain in stationarity, $z_0\sim\pi_d$. Then, as $d\to\infty$,
\begin{equation}
 (S^{(d)},X^{(d)}) \Rightarrow (S, X)
 \qquad\text{in }D([0,\infty),\RR^2),
 \label{eq:joint_process_convergence}
\end{equation}
where $S$ and $X$ are independent stationary OU processes satisfying

\begin{align}
 dS(u)&=-\frac{h(\ell)}{2\sigma_{t^*}^2}S(u)\,du +\sqrt{h(\ell)}\,dB(u),
 \qquad S(0)\sim\N(0,\sigma_{t^*}^2)\,,
 \label{eq:scalar_OU_limit}\\
 dX(u)&=-\frac{s}{2}h_\perp(\ell)X(u)\,du +\sqrt{h_\perp(\ell)}\,dB_\perp(u),
 \qquad X(0)\sim\N(0,s^{-1}).
 \label{eq:transverse_OU_limit}
\end{align}
Here $B$ and $B_\perp$ are independent Brownian motions in $\RR$. The speed functions are

\begin{equation}
    h(\ell)=\left(\frac{\tau}{\delta\kappa}\right)^2\ell^2K(\sigma_\ell)\,,
    \label{eq:h_diffusion}
\end{equation}
and 
\begin{equation}
 h_\perp(\ell)=\ell^2\,2\Phi(-\sigma_\ell/2)=\mathcal E(\ell),
 \label{eq:h_transverse}
\end{equation}
with $\sigma_\ell$ given in \eqref{eq:sigmaell_def} and
\begin{equation}
K(\sigma) =\left(2+\frac{\sigma^2}{2}\right)\Phi(-\sigma/2) -\sigma\phi(\sigma/2)\,.
 \label{eq:K_diffusion}
\end{equation}
\end{thm}

\begin{pf}
The detailed argument is given in Appendix~\ref{app:scalar_diffusion_proof}.
\end{pf}

\begin{rmk}[Properties of $K$]
\label{rmk:K_properties}
We have $K(0)=1$ and $K'(\sigma)=\sigma\Phi(-\sigma/2)-2\phi(\sigma/2)<0$ for $\sigma\ge0$ by Mills' inequality, so $K$ is strictly decreasing. We also have that $K(\sigma) <2\Phi(-\sigma/2)$ for $\sigma>0$ by the same inequality. Considering the different coordinate scales, the speeds satisfy
\begin{equation}
    \frac{h(\ell)}{(\tau/(\delta\kappa))^2h_\perp(\ell)} =\frac{K(\sigma_\ell)}{2\Phi(-\sigma_\ell/2)}<1.
\end{equation}
The leading longitudinal increment and the log acceptance ratio depend on the same Gaussian variable $\xi_1$ and hence cannot be treated as asymptotically independent. In a transverse direction, the leading noise is independent of $\xi_1$, its odd moments against the acceptance indicator vanish, and its acceptance-weighted second moment converges to the acceptance probability. This gives the factor $K$ in $h$ and the factor $2\Phi(-\sigma_\ell/2)$ in $h_\perp$.
\end{rmk}

Since $h_\perp(\ell)=\mathcal E(\ell)$ (with $\mathcal E(\ell)$ defined in \eqref{eq:ESJD_limit}), Theorem~\ref{thm:joint_diffusion_limit} identifies the ESJD criterion with the diffusion-speed criterion for a transverse observable. Its optimal limiting acceptance probability is $0.574$ by \eqref{eq:astar_ellstar}. In contrast, $h(\ell)$ is maximised at $\ell_h=(\sigma_*/(|C_*|\tau))^{1/3}$, where $\sigma_*\approx0.7378$ solves $(1+\sigma^2)\Phi(-\sigma/2)=2\sigma\phi(\sigma/2)$, giving limiting acceptance probability $2\Phi(-\sigma_*/2)\approx0.712$. For $\delta=\kappa=1$, these tunings are $\ell_*\approx2.10$ and $\ell_h\approx1.825$ respectively. Thus the two acceptance optima describe different observables of the same chain. 

The case of different optima is not unique to the sphere we consider here. For stationary MALA on i.i.d. product targets satisfying the assumptions of \citet{roberts-rosenthal:1998}, the classical $0.574$ optimum concerns the diffusion limit of a fixed coordinate. That coordinate contributes negligibly to the total log acceptance ratio, so its leading proposal noise is asymptotically independent of the acceptance. The speed is then the squared step scale multiplied by the limiting acceptance probability, as for our transverse coordinate. Product structure does not give this decoupling for every observable as even for a product target, all coordinates share one Metropolis decision at each step. In the present model, the longitudinal fluctuation is also a fluctuation of the total energy since
\begin{equation}
    \frac{\|z\|^2-d\bigl(\delta^2+\kappa^2+2\delta\kappa t^*\bigr)}{2\sqrt d}= \delta\kappa\,\zeta.
\label{eq:longitudinal_energy}
\end{equation}
Such observables can retain the acceptance coupling even for product targets. For example for MALA on a standard Gaussian target, where the log acceptance ratio is a multiple of the energy increment, it can be shown in the same way that the speed of the energy is maximised at acceptance probability $0.712$, while each fixed coordinate keeps the optimum $0.574$.  

We have that the joint limit establishes the clock for the longitudinal fluctuation and every deterministic transverse projection. It is a path-space limit on fixed time intervals for a chain started in stationarity. It is important to note that it does not establish convergence of the entire state in some common space or provide a mixing-time bound for the entire chain, which is left for future work.

\section{Curvature correction}
\label{sec:geom-aware-corr}

The analysis of Section~\ref{sec:optim-scal-shift} shows that standard tangent-space pCN on the shifted sphere requires a step size $\beta = O(d^{-1/6})$ to maintain a nondegenerate stationary acceptance probability. This analysis used an ambient reference $\theta=1$ and identified a coefficient $C_*$ governing the leading acceptance defect. We now examine how this coefficient depends on the choice of the reference precision $\theta$ in \eqref{eq:h_reference_measure}.

For fixed $\theta>0$, tangent pCN with precision $\theta$ has noise scale $\beta/\sqrt \theta$ and drift coefficient
\begin{equation}
  \frac{\gamma}{\theta} =-\frac12\left(\frac{\beta}{\sqrt \theta}\right)^2
   -\frac \theta8\left(\frac{\beta}{\sqrt \theta}\right)^4 +O(\beta^6).
\label{eq:c_eta_h}
\end{equation}
Comparing with the expansion in Proposition~\ref{prop:drift_family_scaling} gives $\mu=-\theta/8$ and hence $C_*(-\theta/8)= (\theta-s) /4$. Thus $\theta=s$ is the unique reference precision for which the leading coefficient vanishes.

\subsection{The intrinsic precision}
\label{sec:intrinsic-precision}

We start with a statistical interpretation of $s$ and then show the link to the geometry of the manifold. For the unit vector $e_{\perp}$ of Section~\ref{sec:diffusion-limit}, conditional rotational symmetry and Lemma~\ref{lem:concentration_t} give
\begin{equation}
  \Var_{\pi_d}\bigl(\langle z-b,e_{\perp}\rangle\bigr) =\frac{r^2}{d-1}\,\E_{\pi_d}[1-t^2]
  \myrightarrow{} \kappa^2\bigl(1-(t^*)^2\bigr)=\frac1s.
  \label{eq:s_two_forms}
\end{equation}
The last equality follows from $\delta\kappa(1-(t^*)^2)=-t^*$ and the definition of $s$. Since the reference Gaussian $\mu_\theta$ \eqref{eq:h_reference_measure} has variance $\theta^{-1}$ in these directions, choosing $\theta=s$ matches their limiting variance under the target.

The same quantity also has a geometric interpretation. Let us take $S_z u=D_u\nhat(z)$ for the shape operator. The intrinsic Hessian of a smooth ambient potential $U$ is
\begin{equation}
  \hess_{\Mc}U(z) = P_z\nabla^2U(z)P_z -\langle\nabla U(z),\nhat\rangle S_z.
  \label{eq:intrinsic_hessian}
\end{equation}
On the shifted sphere, $S_z=r^{-1}I$ on the tangent space, $\langle z,\nhat\rangle=r+\langle b,\nhat\rangle$, and the Ricci curvature is $\mathrm{Ric}_{\Mc}=(d-2)r^{-2}P_z$. Consequently, under stationarity,
\begin{equation}
  \hess_{\Mc}\tfrac12\|z\|^2+\mathrm{Ric}_{\Mc}=
  \left(\frac{d-2}{r^2}-\frac{\delta}{\kappa}t\right)P_z = \bigl(s+o_p(1)\bigr)P_z,
\label{eq:s_hess_ricci}
\end{equation}
since $r^2=\kappa^2d$, $t\myrightarrow{p}t^*$, and $-\delta t^*/\kappa=\tau^2$. Thus $s$ is the limiting scalar coefficient of the sum of the intrinsic Hessian and the Ricci curvature on the sphere.

\subsection{The intrinsic proposal}
\label{sec:intrinsic-proposal}

We now take $\theta=s$ and propose
\begin{equation}
v \mid z \;\sim\; \N\bigl( \frac{\gamma}{s}\,P_z(b), \frac{\beta^2}{s} P_z \bigr),
  \label{eq:intrinsic_proposal}
\end{equation}
followed by the sphere projection of Proposition~\ref{prop:sphere_backprojection}. The scalar $s$ depends only on the fixed sphere parameters and can be computed once before running the chain. This is \eqref{eq:general_proposal} with $c=\gamma/s$ and $\sigma=\beta/\sqrt s$ and Proposition~\ref{prop:Rc_sphere} gives on $\{\|v\|<r\}$
\begin{equation}
  \Rcancel =
  \left[\frac{1-\rho}{1+\rho} +\frac{\alpha}{(1+\rho)r}
    +\frac{1-\rho}{1+\rho}\frac{\langle\nhat,b\rangle}{sr}\right]g
  +\frac{1-\rho}{1+\rho}\frac{g^2}{2sr^2}\,
\label{eq:Rcancel_intrinsic}
\end{equation}
where the scalar reduction \eqref{eq:scalar_reduction} trivially applies.

\subsection{A fixed-step limit}
\label{sec:fixed-step}

The choice $\theta=s$ cancels the leading term in the small-step expansion. We now ask whether this allows a fixed step size as the dimension grows and how the accepted chain moves. We use the same longitudinal and transverse coordinates $(\zeta,X)$ as in Section~\ref{sec:diffusion-limit}, now for the chain generated by the intrinsic proposal.

Let us keep $\beta\in(0,1)$ fixed as $d\to\infty$. We use the three scalar statistics $(t,\xi_1,W)$ of Section~\ref{sec:reduct-scal-suff}. Throughout this subsection, $J_s$ denotes the coefficient of $g$ in \eqref{eq:Rcancel_intrinsic}. Set
\begin{equation}
  q:=1+\frac{\gamma}{s\kappa^2}\in(0,1).
\label{eq:q_def}
\end{equation}

\begin{thm}[Fixed-step limit]
\label{thm:fixed_step}
Fix $\delta,\kappa>0$ and $\beta\in(0,1)$. Consider the chain generated by the intrinsic proposal \eqref{eq:intrinsic_proposal}, started in stationarity, $z_0\sim\pi_d$. Set $g=J_s=0$ and $\Rcancel=-\infty$ on projection failure.
\begin{enumerate}
\item[(i)] The projection failure probability tends to zero, and
\begin{equation}
  \frac{g}{\sqrt d}\myrightarrow{d}\bar g,
  \qquad
  \sqrt d\,J_s\myrightarrow{d}\bar J_s,
  \qquad
  \Rcancel\myrightarrow{d}\bar R,
  \label{eq:fixed_step_limit}
\end{equation}
jointly. Here $\zeta\sim\N(0,\sigma_{t^*}^2)$ and $\xi_1,\omega\sim\N(0,1)$ are independent, and
\begin{align}
  \bar g &=
  \delta\kappa\frac{\rho^2-q}{q}\,\zeta +\frac{\rho\beta\tau}{q\sqrt s}\,\xi_1 +\frac{\beta^2\tau^2}{\sqrt2\,q\,s}\,\omega,
\label{eq:gbar}\\
  \bar J_s &=
  \frac1s\left[\frac{1-\rho}{1+\rho}\left(\frac{\delta\rho}{\kappa q}\,\zeta +\frac{\beta\tau}{\sqrt s\,\kappa^2q}\,\xi_1\right) +\frac{\gamma}{\sqrt2\,\kappa^2q}\,\omega \right],
\label{eq:Jbar}\\
  \bar R &=
  \bar J_s\bar g +\frac{1-\rho}{1+\rho} \frac{\bar g^2}{2s\kappa^2}.
\label{eq:Rbar}
\end{align}
Consequently, the stationary acceptance probability satisfies
\begin{equation}
  a_d(\beta)\longrightarrow \E[1 \wedge e^{\bar R}]\in(0,1).
  \label{eq:fixed_step_acceptance}
\end{equation}
\item[(ii)] For every fixed $m\in\mathbb N$, the process $(\zeta_k,X_k)_{0\le k\le m}$ converges in law to the stationary Markov chain whose initial law is given by $(\zeta_0,X_0) \sim \N(0,\sigma_{t^*}^2) \otimes \N(0,s^{-1})$ and the following transition: from the current state $(\zeta_k,X_k)$, draw independent $\xi_1,\xi_2, \omega\sim\N(0,1)$ and $u\sim U(0,1)$. Form $\bar g,\bar J_s,\bar R$ using \eqref{eq:gbar}--\eqref{eq:Rbar} with $\zeta=\zeta_k$, and set
\begin{equation}
  (\zeta_{k+1},X_{k+1}) =
  \begin{cases}
    \displaystyle
    \left(\zeta_k+\frac{\bar g}{\delta\kappa},\rho X_k+\frac{\beta}{\sqrt s}\xi_2 \right),
    & \log u<\bar R,\\[1.5ex]
    (\zeta_k,X_k), & \text{otherwise}.
  \end{cases}
\label{eq:limit_kernel}
\end{equation}
The limiting kernel is reversible with respect to its initial law.
\end{enumerate}
\end{thm}

\begin{pf}
Proof is given in Appendix~\ref{app:thm:fixed_step_proof}.
\end{pf}

Theorem~\ref{thm:fixed_step} shows that under stationarity, the intrinsic proposal allows a fixed step size with nondegenerate acceptance and nontrivial limiting dynamics for $(\zeta,X)$.

\begin{rmk}[Dependence on the initial state]
The acceptance probability in Theorem~\ref{thm:fixed_step} is averaged over the stationary distribution. This average has a positive limit, although the conditional acceptance probability at the point $z=b+r\hat b$ tends to zero as $d\to\infty$ at fixed $\beta$. Consequently, at fixed $\beta$, there is no dimension-independent bound on the total-variation mixing time that holds uniformly over initial states. This distinction between stationary and nonstationary acceptance seems analogous to the transient behaviour of MALA for which certain initial conditions require smaller proposal scales than those in stationarity \cite{christensen-roberts-rosenthal:2005,jourdain-lelievre-miasojedow:2014}.
\end{rmk}

\subsection{Curvature corrected tangent pCN}

As a direct consequence, taking the intrinsic proposal with $\theta=s$ as given in~\eqref{eq:intrinsic_proposal} results in a curvature corrected tangent pCN which scales independently of dimension. The changes that need to be made to Algorithm~\ref{alg:tangent-pCN} are: Compute $m_z$ in line 2 (and $m_{z'}$ in line 9) via $m_z \leftarrow \gamma P_z z/\theta$ (and $m_{z'}\leftarrow \gamma P_{z'}z'/\theta$). Compute the tangent increment $v$ in line 3 by $m_z + \beta P_z\xi/\sqrt{\theta}$. Lastly, accept if $\log u < R$ for $R$ computed via
\begin{equation}
    R = -\frac12 \left(\|z'\|^2 - \|z\|^2\right) + \frac \theta{2\beta^2} \left( \|v-m_z\|^2 - \|v' - m_{z'}\|^2\right)\,,
\end{equation}
or equivalently by equation~\eqref{eq:Rcancel_intrinsic}. This is the algorithm used to produce the curvature corrected pCN results in Figure~\ref{fig:sphere-direct-cancel}. The intrinsic precision has to be known or estimated to some accuracy in practice. On the shifted sphere, this is computed explicitly using the known $\delta$ and $\kappa$ which Theorem~\ref{thm:fixed_step} uses. When the intrinsic precision must be estimated, recovering the same acceptance or transition limit is not guaranteed.

\section{Discussion and Outlook}

The shifted sphere analysis gives us a tractable example of how geometry and the choice of Gaussian reference jointly determine the scaling of constrained MCMC. Under stationarity, tangent pCN retains enough Gaussian cancellation to improve the $d^{-1/2}$ scaling of tangent RWMH to $d^{-1/6}$ (except on $s=1$ where fixed steps are admissible). The structure of the remaining defect gives an ESJD-optimal acceptance probability of $0.574$ which is shared by every member of the drift family. For tangent pCN, the joint diffusion limit identifies this ESJD as the transverse diffusion speed, whereas the longitudinal speed is maximised at $0.712$ with coordinates evolving on the $d^{1/3}$ time scale. Finally, matching the reference precision to $s$, characterised by the intrinsic potential Hessian and Ricci curvature, yields nondegenerate stationary acceptance and explicit limiting dynamics at fixed step size.

Sampling a Gaussian measure on a shifted sphere is merely a benchmark problem to derive scaling properties. In fact, it was shown along the way that this particular problem can be reduced to only three scalar statistics, with uniform angles in the complement, which allows us to construct much simpler dimensionally independent samplers for that specific geometry. The actual motivation for constructing high-dimensional manifold samplers is instead the prospect of sampling on arbitrary manifolds. Here, though, demonstrating rigorous scaling properties is much harder.

A glimpse of this can be intuited by looking at equation~\eqref{eq:Rcancel_manifold_residuals} and adding the logarithmic coarea correction, i.e.
\begin{equation*}
    \Rcancel = \frac1{2\beta^2}\left(r_b^2 - r_f^2\right) + \log \left(\frac{\|\nabla F(z)\|}{\|\nabla F(z')\|}\right)\,,
\end{equation*}
with the residuals $r_f$, $r_b$ of \eqref{eq:Rcancel_manifold_residuals}. This formula is true for arbitrary hypersurfaces, where for the sphere further cancellations occur and the coarea term disappears. For a general manifold,
\begin{equation*}
    r_f = \langle\nhat, z'-\rho z\rangle = \langle \nhat, z + v + \alpha \nhat - \rho z\rangle = (1-\rho)\langle \nhat, z\rangle + \alpha\,,
\end{equation*}
and since acceptance is bounded below through $r_f^2$ and the coarea term, we see that a bound on $|r_f|$ that does not grow with $d$ would be sufficient for dimensionally independent scaling of the acceptance probability. Note that such a bound on $r_f$ is not necessary, and in fact untrue for the case of the sphere, where both $\langle\nhat,z\rangle$ and $\operatorname{tr}S_z$ are $O(\sqrt d)$ and $r_f$ is bounded only through the leading order cancellation (uncorrected tangent pCN, Section~\ref{sec:proposed-algorithm}) or full cancellation (curvature corrected tangent pCN, Section~\ref{sec:geom-aware-corr}) of its dimensionally diverging terms with their counterparts in $r_b$. 

Still, a bound on the growth of $r_f$ would be sufficient, and it is worthwhile to consider what properties of the considered manifold could guarantee it. In a small-$\beta$ expansion, $\alpha$ can be approximated by realising that
\begin{align*}
    0 = F(z') &= F(z+v+\alpha\nhat)\\
    &\approx \underbrace{F(z)}_{=0} + \langle \nabla F(z), v + \alpha\nhat\rangle + \tfrac12 (v + \alpha \nhat)^\top \nabla^2 F(z)(v + \alpha \nhat) \\
    &\approx \alpha \|\nabla F(z)\| + \tfrac12 v^\top \nabla^2 F(z) v\,,
\end{align*}
dropping higher order terms in $\beta$. Therefore, and with $v = \beta P_z \xi + O(\beta^2)$, we can write to leading order
\begin{equation*}
    \alpha \approx -\frac{\beta^2}{2} \xi^\top\frac{P_z \nabla^2 F(z) P_z}{\|\nabla F(z)\|}\xi = -\frac{\beta^2}{2} \xi^\top S_z\xi = -\frac{\beta^2}{2} \sum_{i=1}^{d-1} \kappa_i \xi_i^2\,,
\end{equation*}
where $\kappa_i$ are the principle curvatures of the manifold with shape operator $S_z$. Expectation and variance of $\alpha$ are thus controlled by the trace and Hilbert-Schmidt norm of the shape operator, $\operatorname{tr}S_z$ and $\sum_{i=1}^{d-1} \kappa_i^2$, respectively. The acceptance probability on general manifolds are therefore directly related to the normal component and the curvature spectrum.

A particularly relevant problem is that of sampling conditioned diffusions. To sample an It\^o stochastic differential equation (SDE), it is standard practice to discretise the continuous temporal variable, obtaining a finite but high-dimensional sampling problem. If the goal is to condition the SDE on a scalar variable, then we are effectively sampling a Gaussian measure on a codimension one submanifold, and our approaches above become applicable in principle. Following the above arguments, the relevant quantities are the directional derivative of the conditioned variable with respect to the noise, and the second variation of the noise-to-variable map. In the sense of Malliavin calculus~\cite{nualart:2006, inahama:2014}, the first can typically be uniformly bounded, and the second yields a trace-class operator in the case of additive Gaussian noise SDEs~\cite{ben-arous:1988, schorlepp-tong-grafke-etal:2023} (or Hilbert-Schmidt for multiplicative noise~\cite{klose:2026, schorlepp-grafke:2026}). In other words, the manifold defined by a conditioned additive Gaussian diffusion is actually milder than the sphere benchmark: A cancellation of forward and backward geometric residual is not needed here, as the forward residual can be bounded on its own. This means that, in contrast to the sphere geometry, we expect that uncorrected tangent pCN, in the sense of Section~\ref{sec:proposed-algorithm}, should already scale independently of dimension $d$ (i.e. independently of temporal resolution) even without curvature correction, while tangent MALA and RWMH should exhibit degenerating acceptance probabilities in the temporal continuum limit.

\begin{table}
  \caption{Acceptance probabilities of tangent pCN, tangent MALA, and tangent RWMH for varying $\beta$ and $d$ when sampling a Gaussian measure on the family of $(d-1)$-dimensional manifolds given by conditioning a discretised stochastic process onto a scalar constraint.}
  \centering
  \label{tbl:process}
\begin{tabular}{ll rrrrrrrrrrrrrrr}\toprule
& $d$ & 
$2^8$ & 
$2^{10}$ & 
$2^{12}$ & 
$2^{14}$ & 
$2^{16}$ & 
$2^{18}$ & 
$2^{20}$ & 
$2^{22}$\\\midrule
$\beta=0.7$ & pCN & 
$0.987$ & 
$0.987$ & 
$0.988$ & 
$0.985$ & 
$0.986$ & 
$0.986$ & 
$0.987$ & 
$0.987$ \\
& MALA & 
$0.498$ & 
$0.171$ & 
$0.010$ & 
$0.001$ & 
--- & 
--- & 
--- & 
---\\
& RWMH & 
--- & 
--- & 
--- & 
--- & 
--- & 
--- & 
--- & 
---\\
\midrule
$\beta=0.3$ & pCN & 
$0.990$ & 
$0.990$ & 
$0.990$ & 
$0.989$ & 
$0.991$ & 
$0.990$ & 
$0.991$ & 
$0.990$ \\
& MALA & 
$0.956$ & 
$0.913$ & 
$0.834$ & 
$0.662$ & 
$0.398$ & 
$0.105$ & 
$0.007$ & 
$0.002$ \\
& RWMH & 
$0.025$ & 
$0.004$ & 
--- & 
--- & 
--- & 
--- & 
--- & 
---\\\bottomrule
    \end{tabular}
\end{table}

\begin{figure}
    \centering
    \includegraphics[width=0.8\linewidth]{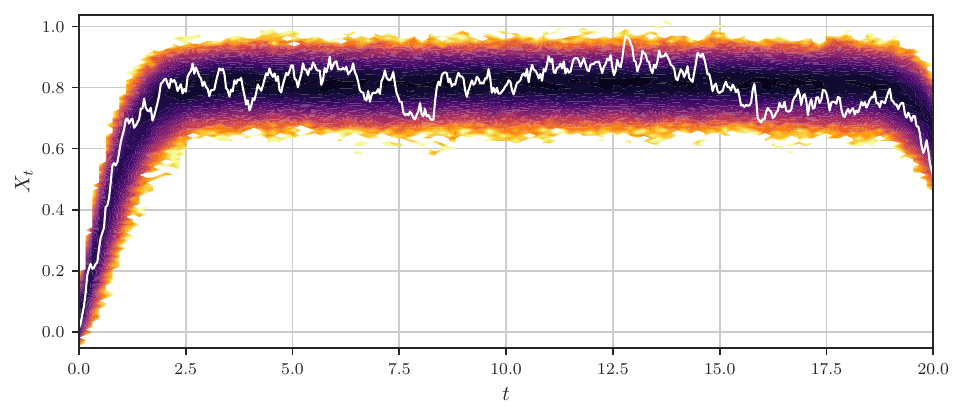}
    \caption{Path ensemble of the discretised stochastic
      process and conditioned on the
      integral constraint~(\ref{eq:process}). Depicted is the path density as shading, and a
      single arbitrary sample trajectory as white line.}
    \label{fig:process}
\end{figure}

This is confirmed by a concrete example: take the nonlinear SDE given by
\begin{equation}
\label{eq:process}
  dX_t = -X_t^3\,dt + \sqrt{\varepsilon} \, dW_t\,,\ X_0=0\,,\  t\in[0,T]\,,\quad\text{conditioned on}\quad \int_0^T X_t^3\,dt = Z\,,
\end{equation}
discretised by Euler-Maruyama scheme and by a left Riemann sum, respectively. Given that we are free to choose a number of timesteps $d$, and generally want to pick $d$ large for high numerical accuracy, this defines a family of manifold sampling problems in the high-dimensional regime. Note that for this problem, the normal vector has a norm dependent on location, so we need to include the coarea term into the algorithm.

Table~\ref{tbl:process} shows the results, confirming the discussed heuristics: For a value $\beta=0.7$, RWMH is unusable already for $d=2^8=256$ grid points, while tangent MALA decays more mildly and remains feasible to thousands of grid points. Tangent pCN has a high acceptance rate in this case independent of dimensionality up to $d=2^{22}\approx 4\cdot 10^6$ grid points. Lowering $\beta$ to $0.3$ reveals RWMH still being unusable even at moderate resolution, while tangent MALA starts to become feasible in high dimension, and tangent pCN again empirically scaling independently of dimension throughout. Figure~\ref{fig:process} depicts an ensemble of trajectories drawn from the conditioned process at $d=512$ sampled via pCN at $\beta=0.7$, and we take $T=20$, $\varepsilon=10^{-2}$ and $Z=10$ for all numerical experiments.

As further outlook, it can be envisioned that the sampling techniques we introduced in this paper can be generally applied to generative models defined as push-forwards of Gaussian reference measures, including for example diffusion/score-based generative models~\cite{song-sohl-dickstein-kingma-etal:2021}, normalizing flows trained by flow matching~\cite{lipman-chen-ben-hamu-etal:2023}, or learned stochastic dynamics~\cite{kidger-foster-li-etal:2021}. This would mean that the generative model could be exactly conditioned on nonlinear conditions, for example fixing hard constraints on the generated output.

%
%
\appendix
\section{Proofs}

\subsection{Proof of Lemma~\ref{lem:concentration_t}}
\label{app:concentration_t_proof}

Let us define the unnormalised log-density
\begin{equation}
\ell_d(t)=\frac{d-3}{2}\log(1-t^2)-\delta\kappa d\,t\,, \qquad t\in(-1,1).
\end{equation}
Since $t^*$ satisfies $-t^*/(1-(t^*)^2)=\delta\kappa$, we have
\begin{equation*}
  \ell_d'(t^*)=-3\delta\kappa,
  \qquad
  \ell_d''(t)=-(d-3)\frac{1+t^2}{(1-t^2)^2}\le-(d-3),
  \qquad
  \frac{\ell_d''(t^*)}{d}\myrightarrow{}-\sigma_{t^*}^{-2}.
\end{equation*}
The density of $\zeta=\sqrt d\,(t-t^*)$ is proportional to $f_d(u):=\exp\{\ell_d(t^*+u/\sqrt d)-\ell_d(t^*)\}$ on
$\bigl(-\sqrt d(1+t^*),\sqrt d(1-t^*)\bigr)$, extended by zero outside its support. Taylor expansion at $t^*$ gives, for every fixed $u$,
\begin{equation*}
  f_d(u)=\exp\Bigl\{-\frac{3\delta\kappa\,u}{\sqrt d} +\frac{\ell_d''(t^*)}{2d}\,u^2+O(d^{-1/2})\Bigr\}
  \myrightarrow{}\exp\Bigl(-\frac{u^2}{2\sigma_{t^*}^2}\Bigr),
\end{equation*}
Consequently, for $d\ge6$,
\begin{equation*}
  f_d(u)\le\exp\Bigl(\frac{3\delta\kappa|u|}{\sqrt d}-\frac{u^2}{4}\Bigr)
  \le C \,e^{-u^2/8}.
\end{equation*}
where $C$ is independent of $d$. By dominated convergence the densities of $\zeta$ converge to $\N(0,\sigma_{t^*}^2)$ and Scheff\'e's lemma gives \eqref{eq:t_CLT}. Since the normalising constants are bounded below for large $d$, the same envelope bounds the density of $\zeta$ by a multiple of $e^{-u^2/8}$ uniformly in $d$, so $\E|\zeta|^k=O(1)$ for every $k$, and $\E_\pi|t-t^*|^k=d^{-k/2}\,\E|\zeta|^k=O(d^{-k/2})$ is \eqref{eq:t_moments}.

\subsection{Refined expansion of \texorpdfstring{$g$}{g}}
\label{app:Lambda_derivation}

Throughout this subsection, $c_\beta\in\RR$ denotes a deterministic drift coefficient and the proposal is
the drift family \eqref{eq:drift_family}. Tangent pCN is the member $c_\beta=\gamma$, which is all that Sections~\ref{sec:expans-at-crit}--\ref{sec:cons-scal-optim} require; we carry a general $c_\beta$ because the proof is identical and the drift family \eqref{eq:drift_family} of Section~\ref{sec:drift-family} then obtains its scaling theory from the same lemma at no extra effort. The scalar reduction \eqref{eq:scalar_reduction} holds verbatim with $\gamma$ replaced by $c_\beta$,
\begin{equation}
  g = c_\beta \|P_z(b)\|^2 + \beta \|P_z(b)\|\,\xi_1 + \alpha\,\langle\nhat,b\rangle,
  \qquad
  \|v\|^2 = (c_\beta \|P_z(b)\|+\beta\xi_1)^2 + \beta^2 W,
  \label{eq:g_split_general}
\end{equation}
with $\xi_1\sim\N(0,1)$, $W\sim\chi^2_{d-2}$, and $(t,\xi_1,W)$
mutually independent. By
Lemma~\ref{lem:concentration_t},
\begin{equation}
  \|P_z(b)\|^2=\tau^2 d+O_p(\sqrt d),
  \quad
  \frac{\langle\nhat,b\rangle}{r}= -\tau^2 +O_p(d^{-1/2})\,,
  \label{eq:blocks_refined}
\end{equation}
while $|\langle\nhat,b\rangle|\le\delta\sqrt d$ deterministically and
$W=(d-2)+O_p(\sqrt d)$.

Finally, for $y\in[0,\tfrac12]$, Taylor's theorem with Lagrange
remainder gives $|\sqrt{1-y}-1+\tfrac y2+\tfrac{y^2}8|\le y^3$. Hence,
whenever $\|v\|^2\le r^2/2$, the projection scalar $\alpha$ of
Proposition~\ref{prop:sphere_backprojection}(i) admits, with
$y=\|v\|^2/r^2$, the third-order expansion
\begin{equation}
  \alpha=-\frac{\|v\|^2}{2r}-\frac{\|v\|^4}{8r^3}
         +\vartheta\,\frac{\|v\|^6}{r^5},
  \qquad |\vartheta|\le1.
  \label{eq:alpha_third_order_app}
\end{equation}
The following event controls $\|v\|^2$ pathwise for the drift family \eqref{eq:drift_family}.

\begin{lem}[The good event $\Omega_d$]
\label{lem:good_event}
Let $\Omega_d := \{\,W + 2\xi_1^2 \le 2d\,\}$, where
$W\sim\chi^2_{d-2}$ and $\xi_1\sim\N(0,1)$ are the proposal statistics
of the scalar reduction. Then there is a constant $c>0$ with
\begin{equation*}
  \mathbb P(\Omega_d^{\complement})\le e^{-cd}
  \qquad\text{for all } d\ge 3 .
\end{equation*}
Moreover, there exist $\beta_0, C_0>0$, depending only on
$(\delta,\kappa)$, such that for every $\beta\le\beta_0$, pathwise on
$\Omega_d$,
\begin{equation}
  \|v\|^2 \le C_0\,\beta^2 d \le \tfrac12\,r^2 ,
  \label{eq:v2_pathwise_bound}
\end{equation}
for any drift proposal $v=c_\beta P_z(b)+\beta P_z\xi$ with $|c_\beta|\le\beta$ with
$C_0 = 2(1+\delta^2)$ and $\beta_0^2 = 1 \wedge \kappa^2/(4(1+\delta^2))$. In particular, on $\Omega_d$ the projection onto $\Mc$ exists and the third-order expansion \eqref{eq:alpha_third_order_app} of $\alpha$ is
valid.
\end{lem}

\begin{pf}
The tail bound:
$\mathbb P(\Omega_d^{\complement})
 \le \mathbb P(W>\tfrac32 d) + \mathbb P(\xi_1^2>\tfrac14 d)
 \le e^{-cd}$
by standard Chernoff bounds for $\chi^2_{d-2}$ and $\chi^2_1$. By $(a+b)^2\le 2a^2+2b^2$, $|c_\beta|\le\beta$ and $\|P_z(b)\|^2\le\delta^2 d$,
\begin{equation*}
  \|v\|^2 \le 2c_\beta^2\|P_z(b)\|^2 + 2\beta^2\xi_1^2 + \beta^2 W
  \le 2\beta^2\delta^2 d + \beta^2\bigl(W+2\xi_1^2\bigr)
  \le 2\beta^2 d\,(1+\delta^2)
\end{equation*}
on $\Omega_d$; the choice of $\beta_0$ gives $\le\kappa^2 d/2 = r^2/2$. The final claim holds since $\|v\|^2/r^2\le\tfrac12$ on $\Omega_d$.
\end{pf}

\begin{lem}[Refined expansion of $g$]
\label{lem:g_refined}
Let $\Omega_d$ and $\beta_0$ be as in Lemma~\ref{lem:good_event}. Assume stationarity and let $\beta=\beta_d\to0$ with $\beta_d\le\beta_0$. Then, on $\Omega_d$:
\begin{enumerate}
\item[(i)] If $c_\beta=-\tfrac12\beta^2+\mu\beta^4+o(\beta^4)$ with
$\mu\in\RR$ fixed, then
\begin{equation}
  g = \beta\tau\sqrt d\,\xi_1
     + \Lambda(\mu)\,\beta^4 d
     + O_p(\beta)
     + O_p(\beta^2\sqrt d)
     + O_p(\beta^3\sqrt d)
     + o(\beta^4 d).
  \label{eq:g_refined_general}
\end{equation}
Here $\Lambda(\mu)$ and $C_*(\mu)$ are defined in \eqref{eq:Lambda_mu_def} and \eqref{eq:Cstar_mu}, respectively. If moreover $c_\beta=-\tfrac12\beta^2+\mu\beta^4+O(\beta^6)$, the final term improves to $O(\beta^6 d)$.
\item[(ii)] If $c_\beta\equiv0$, then
\begin{equation}
  g=\beta\tau\sqrt d\,\xi_1+\tfrac12\tau^2\beta^2 d
    +O_p(\beta)+O_p(\beta^2\sqrt d)+O_p(\beta^4 d).
  \label{eq:g_refined_rw}
\end{equation}
\end{enumerate}
\end{lem}

Two remarks before the proof. First, the tangent pCN case $c_\beta=\gamma$ is Lemma~\ref{lem:g_cancellation}; its uniform in $\beta_d$ form is what the trichotomy proof in Appendix~\ref{app:scaling_trichotomy_proof} requires. Second, the constant $\Lambda(\mu)$ is affine in $\mu$ with slope $\tau^2$ (the drift acts through $c_\beta \|P_z(b)\|^2$ alone) and intercept $\tau^2 s/8$ split between the two curvature mechanisms identified in the proof: $\tau^4/8$ from the drift-squared part of $\|v\|^2$ meeting the leading projection term, and $\tau^2/(8\kappa^2)$ from the $\beta^2 d$ part of $\|v\|^2$ meeting the second-order projection term.

\begin{pf}
By Lemma~\ref{lem:good_event} which is applicable since $|c_\beta|\le\beta$ for all sufficiently small $\beta$, $\mathbb P(\Omega_d^{\complement})\le e^{-cd}$ and, pathwise on $\Omega_d$, $\|v\|^2\le C_0\beta^2 d\le r^2/2$, so the expansion \eqref{eq:alpha_third_order_app} of $\alpha$ is valid on $\Omega_d$.

\emph{Step 1 (expansion of $\|v\|^2$).}
Since $c_\beta^2=\tfrac14\beta^4+O(\beta^6)$, expanding
\eqref{eq:g_split_general} with \eqref{eq:blocks_refined} gives
\begin{equation}
  \|v\|^2=\beta^2 d+\frac{\tau^2}{4}\beta^4 d
   +O_p(\beta^2\sqrt d)+O_p(\beta^3\sqrt d)+O(\beta^6 d),
  \label{eq:v2_refined}
\end{equation}
where $O_p(\beta^2\sqrt d)$ absorbs $\beta^2(W-(d-2))$, $-2\beta^2$,
$\beta^2\xi_1^2$ and $\tfrac14\beta^4(\|P_z(b)\|^2-\tau^2 d)=O_p(\beta^4\sqrt d)$,
while $O_p(\beta^3\sqrt d)$ is the cross term $2c_\beta \|P_z(b)\|\beta\xi_1$
and $O(\beta^6 d)$ collects the deterministic tail of $c_\beta^2 \|P_z(b)\|^2$.

\emph{Step 2 (assembling $g$, case (i)).}
Substituting \eqref{eq:alpha_third_order_app} into
\eqref{eq:g_split_general},
\begin{equation}
  g=c_\beta \|P_z(b)\|^2+\beta \|P_z(b)\|\xi_1
    -\frac{\langle\nhat,b\rangle}{2r}\,\|v\|^2
    -\frac{\langle\nhat,b\rangle}{8r}\,\frac{\|v\|^4}{r^2}
    +\vartheta\,\frac{\langle\nhat,b\rangle\|v\|^6}{r^5}.
  \label{eq:g_assembled}
\end{equation}
We expand each term, using throughout that $|\langle\nhat,b\rangle/r|\le\delta/\kappa$
is a bounded multiplier.

(a) $\beta \|P_z(b)\|\xi_1=\beta\tau\sqrt d\,\xi_1+O_p(\beta)$ by
\eqref{eq:blocks_refined}.

(b) \emph{Cancellation at order $\beta^2 d$.} The contributions of $c_\beta \|P_z(b)\|^2$ and
$-(\langle\nhat,b\rangle/2r)\|v\|^2$ at this order cancel. Indeed, \eqref{eq:v2_refined} gives
\begin{equation}
  -\frac{\beta^2}{2}
  \left[
    \bigl(\|P_z(b)\|^2- \tau^2d\bigr)
    +d\left(\frac{\langle\nhat,b\rangle}{r}+\tau^2\right)\right]
  = O_p(\beta^2\sqrt d),
  \label{eq:cancellation_quantified}
\end{equation}
by the two estimates in \eqref{eq:blocks_refined}. This is the cancellation of
Lemma~\ref{lem:g_cancellation}.

(c) \emph{Drift contribution at order $\beta^4 d$.}
$\mu\beta^4 \|P_z(b)\|^2=\mu\tau^2\beta^4 d+O_p(\beta^4\sqrt d)$, and the
$o(\beta^4)$ tail of $c_\beta$ contributes
$o(\beta^4)\cdot \|P_z(b)\|^2=o(\beta^4 d)$ (deterministically, since
$\|P_z(b)\|^2\le\delta^2 d$); under the stronger hypothesis this is
$O(\beta^6 d)$.

(d) \emph{First curvature contribution.} The
$\tfrac{\tau^2}4\beta^4 d$ part of \eqref{eq:v2_refined}, multiplied
by $-\langle\nhat,b\rangle/(2r)$ yields
$\tfrac{\tau^4}8\beta^4 d+O_p(\beta^4\sqrt d)$; the remainders of
\eqref{eq:v2_refined} stay $O_p(\beta^2\sqrt d)+O_p(\beta^3\sqrt d)
+O(\beta^6 d)$ under the bounded multiplier.

(e) \emph{Second curvature contribution.} Squaring
\eqref{eq:v2_refined} and dividing by $r^2=\kappa^2 d$,
\begin{equation*}
  \frac{\|v\|^4}{r^2} =\frac{\beta^4 d}{\kappa^2} +\frac{\tau^2}{2\kappa^2}\beta^6 d +\frac{\tau^4}{16\kappa^2}\beta^8 d + O_p(\beta^4\sqrt d)+O_p(\beta^5\sqrt d)+O(\beta^8 d)+O_p(\beta^4),
\end{equation*}
hence

\begin{align*}
  -\frac{\langle\nhat,b\rangle}{8r}\,\frac{\|v\|^4}{r^2}
  &=\Bigl(\frac{\tau^2}{8}+O_p(d^{-1/2})\Bigr) \frac{\beta^4 d}{\kappa^2} +O_p(\beta^2\sqrt d)+O(\beta^6 d)\\
  &=\frac{\tau^2}{8\kappa^2}\,\beta^4 d +O_p(\beta^2\sqrt d)+O_p(\beta^4\sqrt d)+O(\beta^6 d).
\end{align*}

(f) \emph{Cubic remainder.}
On $\Omega_d$, by \eqref{eq:v2_pathwise_bound}, pathwise
\begin{equation*}
  |\langle\nhat,b\rangle|\,\frac{\|v\|^6}{r^5}
  \le \delta\sqrt d\;\frac{C_0^3\beta^6 d^3}{\kappa^5 d^{5/2}}
  = O(\beta^6 d) = o(\beta^4 d),
\end{equation*}
now a deterministic bound.

Summing (a)--(f) and consolidating
$O_p(\beta^4\sqrt d),O_p(\beta^2)\subset O_p(\beta^2\sqrt d)$ (valid
for $\beta\le1$, $d\ge1$) and $O(\beta^6 d)\subset o(\beta^4 d)$
gives \eqref{eq:g_refined_general} with
$\Lambda(\mu)=\mu\tau^2+\tfrac{\tau^4}8+\tfrac{\tau^2}{8\kappa^2}
=\tau^2(\mu+s/8)$.

\emph{Step 3 (case (ii)).}
With $c_\beta\equiv0$, $\|v\|^2=\beta^2(W+\xi_1^2)
=\beta^2 d+O_p(\beta^2\sqrt d)$, and \eqref{eq:g_assembled} becomes
\begin{equation*}
  g=\beta \|P_z(b)\|\xi_1
   +\Bigl(\frac{\tau^2}{2}+O_p(d^{-1/2})\Bigr)
    \bigl(\beta^2 d+O_p(\beta^2\sqrt d)\bigr)
   +O_p(\beta^4 d)+O_p(\beta^6 d),
\end{equation*}
the $O_p(\beta^4 d)$ bounding the $\|v\|^4$ term and the
$O_p(\beta^6 d)$ the cubic remainder; no cancellation occurs because
the $c_\beta \|P_z(b)\|^2$ term is absent. With (a), this is
\eqref{eq:g_refined_rw}.
\end{pf}

\begin{lem}[Expansion of the defect coefficients]
\label{lem:coeff_general}
Under the hypotheses of Lemma~\ref{lem:g_refined}(i), on the event
$\Omega_d$ of Lemma~\ref{lem:good_event}, for every $\beta_d\to0$,
\begin{equation}
  -1-\frac{2c_\beta}{\beta^2}=-2\mu\beta^2+o(\beta^2),
  \qquad
  -\frac{c_\beta}{\beta^2}\,\frac{\alpha}{r}
   =-\frac{\beta^2}{4\kappa^2}+o_p(\beta^2),
  \qquad
  \frac{c_\beta^2}{\beta^2}\,\frac{\langle\nhat,b\rangle}{r}
   =-\frac{\tau^2}{4}\beta^2+o_p(\beta^2),
  \label{eq:coeff_general}
\end{equation}
and $\dfrac{c_\beta^2}{2\beta^2}\,\dfrac{g^2}{r^2}
=O_p\bigl(\beta^4(1+\beta^6 d)\bigr)$.
\end{lem}

\begin{pf}
The first equation is immediate from the hypothesis. For the second, \eqref{eq:alpha_third_order_app} and \eqref{eq:v2_refined} give $\alpha/r=-\|v\|^2/(2r^2)+O_p(\|v\|^4/r^4)=-\beta^2/(2\kappa^2)+O_p(\beta^2 d^{-1/2})+O(\beta^4)$, and $c_\beta/\beta^2=-\tfrac12+O(\beta^2)$. For the third, $c_\beta^2/\beta^2=\tfrac14\beta^2+O(\beta^4)$ and \eqref{eq:blocks_refined}. For the last,
$g=O_p(\beta\sqrt d)+O(\beta^4 d)$ by Lemma~\ref{lem:g_refined}(i), so $g^2/r^2=O_p(\beta^2)+O(\beta^8 d)$, and $c_\beta^2/(2\beta^2)=O(\beta^2)$.
\end{pf}

\subsection{Proof of Theorem~\ref{thm:joint_diffusion_limit}}
\label{app:scalar_diffusion_proof}

We present here the full proof of Theorem~\ref{thm:joint_diffusion_limit}. We first derive the conditional moments of the longitudinal increment and then use the same expansions to obtain the transverse and mixed moments. A single two-dimensional diffusion approximation completes the proof.

\subsubsection{Locally uniform conditional expansion}
\label{app:conditional_expansion_diffusion}

At leading order, we have that the accepted scalar increment has zero conditional mean and the drift of the limiting diffusion is therefore determined by the next order term. We introduce the following notation to identify that contribution. Set
\begin{equation}
    \varepsilon:=d^{-1/6}, \qquad \widehat W_d:=\frac{W-(d-2)}{\sqrt{2(d-2)}}\,.
    \label{eq:diffusion_bookkeeping}
\end{equation}
Now define
\begin{align}
  g_0 &:= \ell\tau\left(\xi_1 - \frac{C_*\tau\ell^3}{2}\right),
  & R_0 &:= C_*\ell^2 g_0\,,
    \label{eq:g0_R0_diffusion} \\
  g_{1,d} &:= -\frac{\ell^2\tau^2}{2\delta\kappa\sigma_{t^*}^2}\,\zeta
         + \frac{\ell^2\tau^2}{\sqrt{2}}\,\widehat W_d\,,
  & R_{1,d} &:= C_*\ell^2 g_{1,d}\,.
    \label{eq:G1_R1_diffusion}
\end{align}

For each $\zeta$ satisfying $t^*+\zeta / \sqrt d\in[-1,1]$, let $\E_\zeta$ denote expectation over one step of the scalar transition kernel started from the current $\zeta$; equivalently, we write
$\E_\zeta[\cdot]=\E[\cdot\mid\zeta_k=\zeta]$. The next lemma gives the sphere-specific expansions needed in the diffusion proof.

\begin{lem}[Locally uniform conditional expansion]
\label{lem:conditional_expansion_diffusion}
Let $\beta=\ell \varepsilon$ and assume $C_*\ne0$. For every fixed $L>0$, every integer $m \ge 1$ and all sufficiently large $d$ (so that $\beta\le\beta_0$), on the event $\Omega_d$ of Lemma~\ref{lem:good_event},
\begin{align}
    g &= \varepsilon^{-2}g_0+\varepsilon^{-1}g_{1,d}+g_{2,d},
    \label{eq:g_conditional_expansion}\\
    \Rcancel &=R_0+\varepsilon R_{1,d}+R_{2,d},
    \label{eq:R_conditional_expansion}
\end{align}
where
\begin{equation}
    \sup_{|\zeta|\le L} \E_\zeta[|g_{2,d}|^m\mathds1_{\Omega_d}]=O(1),
    \qquad
 \sup_{|\zeta|\le L}
 \E_\zeta[|R_{2,d}|^m\mathds1_{\Omega_d}]
 =O(\varepsilon^{2m}).
 \label{eq:conditional_remainder_bounds}
\end{equation}
Moreover, for every fixed $m$,
\begin{equation}
    \sup_{\zeta}
 \E_\zeta\!\left[
 |\Delta\zeta|^m\mathds1_{\Omega_d^{\complement}}\right]
 \le 2^m d^{m/2}e^{-cd}.
 \label{eq:bad_event_increment_bound}
\end{equation}
\end{lem}

\begin{pf}
Fix $L>0$ and $1\le m<\infty$. For $|\zeta|\le L$, we can write $t=t^*+\zeta\varepsilon^3$. All bounds below hold uniformly over this set. We use the event $\Omega_d$ of Lemma~\ref{lem:good_event}. It depends only on $(\xi_1,W)$ and satisfies
$\mathbb P(\Omega_d^{\complement})\le e^{-cd}$ with $c>0$. For all
sufficiently large $d$, the critical step size $\beta=\ell\varepsilon$ is below the constant $\beta_0$ of that lemma; hence, pathwise on $\Omega_d$, $\|v\|^2\le r^2/2$ and the third-order expansion
\eqref{eq:alpha_third_order_app} is valid.

In the notation of \eqref{eq:diffusion_bookkeeping}, the identities \eqref{eq:nb_expansion} and \eqref{eq:P_expansion} exactly read
\begin{equation}
    \|P_z(b)\|^2=\tau^2\varepsilon^{-6}
        +2\delta\kappa\tau^2\zeta\varepsilon^{-3}
        -\delta^2\zeta^2,
 \qquad
 \frac{\langle\nhat,b\rangle}{r}
   =-\tau^2+\frac{\delta\zeta}{\kappa}\varepsilon^3\,.
 \label{eq:pz2_nb_conditional_blocks}
\end{equation}
Taylor expansion of the square root, uniformly for $|\zeta|\le L$, gives in addition
\begin{equation}
\|P_z(b)\| =\tau\varepsilon^{-3} +\delta\kappa\tau\zeta +O(\varepsilon^3).
\label{eq:pz_conditional_blocks}
\end{equation}
We also have
\begin{align}
\gamma &=-\tfrac12\ell^2\varepsilon^2 -\tfrac18\ell^4\varepsilon^4 + O(\varepsilon^6)\,,
\label{eq:gamma_conditional_blocks}\\
W &=\varepsilon^{-6}-2 +\sqrt2\varepsilon^{-3}\sqrt{1-2\varepsilon^6}\,\widehat W_d.
\label{eq:W_conditional_blocks}
\end{align}
For every fixed $p<\infty$, $\sup_{d\ge3}\E|\widehat W_d|^p<\infty$. For convenience, set $V_d:=\|v\|^2$. Then expanding the exact scalar reduction \eqref{eq:scalar_reduction} and using \eqref{eq:pz_conditional_blocks}--\eqref{eq:W_conditional_blocks} gives
\begin{equation}
V_d=\ell^2\varepsilon^{-4} +\frac{1}{4}\ell^4\tau^2\varepsilon^{-2}
     +\sqrt2\ell^2\varepsilon^{-1}\widehat W_d+\widetilde V_d,
 \qquad \sup_{|\zeta|\le L} \E_\zeta\!\left[|\widetilde V_d|^p\right] = O(1)
 \label{eq:V_conditional_expansion}
\end{equation}
for every fixed $p$. On $\Omega_d$, substituting \eqref{eq:alpha_third_order_app} into
the scalar representation of $g$ gives
\begin{equation}
  g =\gamma\|P_z(b)\|^2+\beta\|P_z(b)\|\xi_1
    -\frac{\langle\nhat,b\rangle}{2r}V_d -\frac{\langle\nhat,b\rangle}{8r}\frac{V_d^2}{r^2}
    +\vartheta\frac{\langle\nhat,b\rangle}{r} \frac{V_d^3}{r^4},
  \qquad |\vartheta|\le1.
\label{eq:g_conditional_assembled}
\end{equation}
The terms are expanded as in the proof of Lemma~\ref{lem:g_refined} and so the first four terms in \eqref{eq:g_conditional_assembled} satisfy
\begin{align}
  \gamma\|P_z(b)\|^2
  &=-\tfrac12\ell^2\tau^2\varepsilon^{-4}
    -\tfrac18\ell^4\tau^2\varepsilon^{-2}
    -\ell^2\delta\kappa\tau^2\zeta\varepsilon^{-1}
    +\mathcal E_{1,d},
    \label{eq:g_conditional_term1}\\
  \beta\|P_z(b)\|\xi_1
  &=\ell\tau\xi_1\varepsilon^{-2}
    +\mathcal E_{2,d},
    \label{eq:g_conditional_term2}\\
  -\frac{\langle\nhat,b\rangle}{2r}V_d
  &=\tfrac12\ell^2\tau^2\varepsilon^{-4}
    +\tfrac18\ell^4\tau^4\varepsilon^{-2} \notag\\
  &\quad
    +\left(
       -\frac{\ell^2\delta}{2\kappa}\zeta
       +\frac{\ell^2\tau^2}{\sqrt2}\widehat W_d
     \right)\varepsilon^{-1}
    +\mathcal E_{3,d},
    \label{eq:g_conditional_term3}\\
  -\frac{\langle\nhat,b\rangle}{8r}\frac{V_d^2}{r^2}
  &=\frac{\ell^4\tau^2}{8\kappa^2}\varepsilon^{-2}
    +\mathcal E_{4,d}\,,
  \label{eq:g_conditional_term4}
\end{align}
where, for $j\in\{1,\ldots,4\}$ the $\mathcal E_{j,d}$ capture all positive powers of $\varepsilon$, so that, for every fixed $p$,
\begin{equation}
\sup_{|\zeta|\le L} \E_\zeta\!\left[|\mathcal E_{j,d}|^p\mathds1_{\Omega_d}\right]
  =O(1).
  \label{eq:g_conditional_component_remainders}
\end{equation}
Finally, Lemma~\ref{lem:good_event} gives $V_d\le C_0\ell^2\varepsilon^{-4}$ on
$\Omega_d$, so for the last term \eqref{eq:g_conditional_assembled} we have the pathwise bound
\begin{equation}
\left|\vartheta\frac{\langle\nhat,b\rangle}{r} \frac{V_d^3}{r^4}\right|
 \le \frac{\delta}{\kappa}\, \frac{C_0^3\ell^6\varepsilon^{-12}}{\kappa^4\varepsilon^{-12}} = O(1).
 \label{eq:g_conditional_cubic}
\end{equation}

The order $\varepsilon^{-4}$ terms in \eqref{eq:g_conditional_term1} and
\eqref{eq:g_conditional_term3} cancel. At order $\varepsilon^{-2}$, the remaining coefficient is
\begin{equation}
\ell\tau\xi_1 +\tfrac18 \ell^4\tau^2 (s-1) = \ell\tau\xi_1+\Lambda\ell^4 = g_0\,,
\end{equation}
using $\Lambda$ as defined in \eqref{eq:g_critical_final}. At order $\varepsilon^{-1}$, the coefficient is
\begin{equation}
-\ell^2\delta\left(\kappa\tau^2+\frac{1}{2\kappa}\right)\zeta +\frac{\ell^2\tau^2}{\sqrt2}\,\widehat W_d = g_{1,d}.
\end{equation}
Defining $g_{2,d}$ by the exact remainder in \eqref{eq:g_conditional_expansion}, \eqref{eq:g_conditional_component_remainders} and \eqref{eq:g_conditional_cubic} prove the first bound in
\eqref{eq:conditional_remainder_bounds}.

On $\Omega_d$, use the split $\Rcancel=Jg+Q$ of \eqref{eq:R_JQ_decomposition}.  We require the leading term of $J$ with an $O(\varepsilon^4)$ remainder; in particular, there is no order $\varepsilon^3$ term. Substituting \eqref{eq:V_conditional_expansion} into the expansion \eqref{eq:alpha_third_order_app} of $\alpha / r$ and using $r^2=\kappa^2\varepsilon^{-6}$ gives
\begin{equation}
\frac{\alpha}{r} =-\frac{\ell^2}{2\kappa^2}\varepsilon^2+\widetilde A_d,
 \qquad \sup_{|\zeta|\le L} \E_\zeta\!\left[|\widetilde A_d|^p \mathds1_{\Omega_d} \right] = O(\varepsilon^{4p})\,,
 \label{eq:alpha_r_conditional}
\end{equation}
for every fixed $p$. The $\varepsilon^{-2}$ correction in $V_d$ and the leading term $V_d^2$ both contribute at order $\varepsilon^4$ to $\alpha/r$, all remaining terms are smaller. Moreover, the
order $\varepsilon^3$ correction in $\langle\nhat ,b\rangle/r$ is multiplied in $J$ by $(1-\rho)/(1+\rho)=\tfrac14\ell^2\varepsilon^2+O(\varepsilon^4)$ and therefore first contributes at order $\varepsilon^5$. Together with $1/(1+\rho)=\tfrac12+O(\varepsilon^2)$, these give
\begin{equation}
    J=C_*\ell^2\varepsilon^2 + \widetilde J_d, \qquad \sup_{|\zeta|\le L} \E_\zeta\!\left[|\widetilde J_d|^p \mathds1_{\Omega_d} \right] = O(\varepsilon^{4p})\,.
    \label{eq:J_conditional_expansion}
\end{equation}
From the expansion of $g$ and the previous remainder bounds, we have $\sup_{|\zeta|\le L} \E_\zeta\!\left[
|g|^p\mathds1_{\Omega_d} \right]=O(\varepsilon^{-2p})$. Hence
\begin{equation}
\sup_{|\zeta|\le L} \E_\zeta\!\left[
|Q|^p\mathds1_{\Omega_d} \right] =O(\varepsilon^{4p})\,.
\end{equation}
Moreover, since $R_0 + \varepsilon R_{1,d} = C_*\ell^2\varepsilon^2 \left( \varepsilon^{-2}g_0 + \varepsilon^{-1}g_{1,d} \right)$, the exact remainder in \eqref{eq:R_conditional_expansion} is
\begin{equation}
    R_{2,d} = C_*\ell^2\varepsilon^2g_{2,d} +\widetilde J_d\,g+Q.
\label{eq:R2_exact_decomposition}
\end{equation}
The first term has $m$-th moment $O(\varepsilon^{2m})$, the second has the same bound by H\"older and \eqref{eq:J_conditional_expansion}, and the last term has $m$-th $O(\varepsilon^{4m})$. This proves the second bound in \eqref{eq:conditional_remainder_bounds}.

For the final bound we have that the accepted chain always has $t_k,t_{k+1}\in[-1,1]$; a failed projection
is a rejection. Therefore, $|\Delta\zeta|=\sqrt d\,|t_{k+1}-t_k|\le2\sqrt d$ deterministically. By Lemma~\ref{lem:good_event}, $\mathbb P(\Omega_d^{\complement})\le e^{-cd}$, and this bound is independent of the current $\zeta$. Therefore,
\begin{equation}
    \sup_\zeta\E_\zeta\!\left[ |\Delta\zeta|^m\mathds1_{\Omega_d^{\complement}} \right]
    \le (2\sqrt d)^m \mathbb P(\Omega_d^{\complement}) \le
    2^m d^{m/2}e^{-cd}.
\end{equation}
This proves \eqref{eq:bad_event_increment_bound}.
\end{pf}

\subsubsection{Gaussian acceptance and conditional moments}
Write $\mathcal A(x):=1\wedge e^x$. The accepted scalar increment and the leading log acceptance ratio are both functions of $\xi_1$, so their moments cannot be factorised. The following identities account for that dependence.

\begin{lem}[Gaussian acceptance identities]
\label{lem:acceptance_moments_diffusion}
Let $\xi_1\sim\N(0,1)$, assume $C_*\ne0$, and let $R_0$ be as in \eqref{eq:g0_R0_diffusion}. Then
\begin{align}
    \E[\mathcal A(R_0)]&= 2\Phi(-\sigma_\ell/2),
\label{eq:A_moment_diffusion}\\
    \E\!\left[\left(\xi_1-\frac{C_*\tau\ell^3}{2}\right) \mathcal A(R_0) \right]&=0,
\label{eq:A_first_moment_disappears}\\
    \E\!\left[\left(\xi_1-\frac{C_*\tau\ell^3}{2}\right)^2 \mathcal A(R_0) \right] &=K(\sigma_\ell),
\label{eq:K_moment_diffusion}\\
\E\!\left[\left(\xi_1-\frac{C_*\tau\ell^3}{2}\right)e^{R_0}\mathds1_{\{R_0<0\}}\right]
&= \frac{C_*\tau\ell^3}{2}\Phi(-\sigma_\ell/2) -\operatorname{sgn}(C_*)\phi(\sigma_\ell/2)\,,
\label{eq:derivative_moment_diffusion}
\end{align}
where $K$ is defined in \eqref{eq:K_diffusion}. In particular, $K(\sigma_\ell)>0$.
\end{lem}

\begin{pf}
Set $a:=C_*\tau\ell^3$ so that $R_0=a\xi_1-a^2 /2$ and $\sigma_\ell=|a|$. The transformation $(\xi_1,a)\mapsto(-\xi_1,-a)$ preserves the law of $\xi_1$ and leaves $R_0$ unchanged, while $\xi_1-\frac a2 \,\mapsto\, -\left(\xi_1-\frac a2\right)$. So identities involving the square of this quantity are unchanged and the identities involving its first power change sign on both
sides. It therefore suffices to treat $a>0$, and then $\sigma_\ell=a$.

For $a>0$, $R_0<0$ exactly when $\xi_1<a/2$ and so $\mathcal A(R_0) = \mathds1_{\{\xi_1\ge a/2\}} +e^{R_0}\mathds1_{\{\xi_1<a/2\}}$. Completing the square gives $e^{ax-a^2/2}\phi(x)=\phi(x-a)$. Consequently, for each of the polynomial functions $F$ used below,
\begin{equation}
\E\!\left[F(\xi_1)\mathcal A(R_0)\right] = \int_{a/2}^{\infty}F(x)\phi(x)\,dx +\int_{-\infty}^{-a/2}F(y+a)\phi(y)\,dy,
\label{eq:acceptance_master}
\end{equation}
where the second integral follows from the substitution $y=x-a$. Taking $F\equiv1$ in \eqref{eq:acceptance_master} proves \eqref{eq:A_moment_diffusion}. Taking $F(x)=x-a/2$ and substituting $y=-x$ in the second integral of \eqref{eq:acceptance_master}, the two integrals cancel and so we have $\E\!\left[\left(\xi_1-\frac a2\right)\mathcal A(R_0)\right]=0$, showing~\eqref{eq:A_first_moment_disappears}.
For \eqref{eq:K_moment_diffusion}, set $c:=a/2$ and take
$F(x)=(x-c)^2$ in \eqref{eq:acceptance_master}. By symmetry of $\phi$,
the two integrals are equal, and using truncated Gaussian moments we get
\begin{align*}
\E\!\left[\left(\xi_1-\frac a2\right)^2\mathcal A(R_0)\right] &= 2\int_c^\infty(x-c)^2\phi(x)\,dx\\
&=2\bigl[(1+c^2)\Phi(-c)-c\phi(c)\bigr]\\
&=\left(2+\frac{\sigma_\ell^2}{2}\right) \Phi(-\sigma_\ell/2) -\sigma_\ell\phi(\sigma_\ell/2) = K(\sigma_\ell)\,.
\end{align*}
The random variable $(\xi_1-a/2)^2\mathcal A(R_0)$ is nonnegative and a.s.~not zero, so $K(\sigma_\ell)>0$.
Finally, the same change of variables gives
\begin{equation}
\E\!\left[\left(\xi_1-\frac a2\right)e^{R_0}\mathds1_{\{R_0<0\}}\right]= \int_{-\infty}^{-a/2}\left(y+\frac a2\right)\phi(y)\,dy = \frac a2\Phi(-a/2)-\phi(a/2).
\end{equation}
For general sign of $C_*$ the symmetry above gives a factor $\operatorname{sgn}(C_*)$ on the second term which proves \eqref{eq:derivative_moment_diffusion}.
\end{pf}

\begin{prop}[Conditional drift, variance, and large jumps]
\label{prop:conditional_moments_diffusion}
As $d\to\infty$, for every $L>0$,
\begin{align}
    \sup_{|\zeta|\le L}\left|d^{1/3}\E_\zeta[\Delta\zeta]+\frac{h(\ell)}{2\sigma_{t^*}^2}\zeta\right| &\myrightarrow{} 0\,,
\label{eq:conditional_drift_diffusion}\\
    \sup_{|\zeta|\le L} \left|d^{1/3}\E_\zeta[(\Delta\zeta)^2]-h(\ell)\right| &\myrightarrow{} 0\,,
\label{eq:conditional_variance_diffusion}\\
    \sup_{|\zeta|\le L}d^{1/3}\E_\zeta[(\Delta\zeta)^4] &\myrightarrow{} 0\,,
\label{eq:conditional_fourth_diffusion}
\end{align}
where $h$ is defined in \eqref{eq:h_diffusion}.
\end{prop}

\begin{pf}
By \eqref{eq:bad_event_increment_bound}, it suffices to work on $\Omega_d$. For every $H$ measurable with respect to $(\xi_1,W)$, integrating over $U$ gives $\E[H\mathds1_{\mathrm{acc}}\mid\xi_1,W] = H\mathcal A(\Rcancel)$. The function $\mathcal A$ is Lipschitz, has a.e. derivative $\mathcal A'(x)=e^x\mathds1_{\{x<0\}}$. Using \eqref{eq:R_conditional_expansion}, it satisfies
\begin{align}
&\left| \mathcal A(\Rcancel)-\mathcal A(R_0) -(\varepsilon R_{1,d}+R_{2,d})\mathcal A'(R_0) \right| \notag\\
& \qquad \le \tfrac12|\varepsilon R_{1,d}+R_{2,d}|^2 +2|\varepsilon R_{1,d}+R_{2,d}| \mathds1_{\{|R_0|\le|\varepsilon R_{1,d}+R_{2,d}|\}} .
\label{eq:acceptance_kink_bound}
\end{align}
Since $R_0=C_*\ell^2g_0$, on the event in the second term of \eqref{eq:acceptance_kink_bound} we have $|g_0|\le |\varepsilon R_{1,d}+R_{2,d}|/(|C_*|\ell^2)$. Lemma~\ref{lem:conditional_expansion_diffusion} therefore justifies a first order expansion inside the drift and gives uniformly for $|\zeta|\le L$,
\begin{equation}
\E_\zeta[g\mathcal A(\Rcancel)] = \varepsilon^{-2}\E[g_0\mathcal A(R_0)] + \varepsilon^{-1}\left\{ \E[g_{1,d}\mathcal A(R_0)] +\E[g_0R_{1,d}e^{R_0}\mathds1_{\{R_0<0\}}]\right\} +O(1),
\label{eq:accepted_drift_expansion}
\end{equation}
where all remaining products are of smaller order. The first term in \eqref{eq:accepted_drift_expansion} vanishes by equation~\eqref{eq:A_first_moment_disappears}. For the braces term in \eqref{eq:accepted_drift_expansion} we have
\begin{align*}
&\E[g_{1,d}\mathcal A(R_0)] +\E[g_0R_{1,d}e^{R_0}\mathds1_{\{R_0<0\}}]\\
&\quad= -\frac{\ell^2\tau^2}{2\delta\kappa\sigma_{t^*}^2}\,
 \zeta\Biggl\{\E[\mathcal A(R_0)]+C_*\tau\ell^3
 \E\!\left[\left(\xi_1-\frac{C_*\tau\ell^3}{2}\right)e^{R_0}\mathds1_{\{R_0<0\}}\right]\Biggr\}\\
 &\quad=-\frac{\ell^2\tau^2}{2\delta\kappa\sigma_{t^*}^2}\,\zeta
 \left[\left(2+\frac{\sigma_\ell^2}{2}\right)\Phi\!\left(-\frac{\sigma_\ell}{2}\right)
 -\sigma_\ell\phi\!\left(\frac{\sigma_\ell}{2}\right)\right]\\
 &\quad=-\frac{\ell^2\tau^2}{2\delta\kappa\sigma_{t^*}^2}\,\zeta K(\sigma_\ell).
\end{align*}
Here the second equality follows from~\eqref{eq:A_moment_diffusion} and~\eqref{eq:derivative_moment_diffusion}, with $C_*\tau\ell^3 \operatorname{sgn}(C_*) = |C_*|\tau\ell^3=\sigma_\ell$. Using \eqref{eq:delta_zeta}, $d^{1/3}=\varepsilon^{-2}$, and \eqref{eq:h_diffusion} proves \eqref{eq:conditional_drift_diffusion}.
For the second moment, Lemma~\ref{lem:conditional_expansion_diffusion},
the Lipschitz property of $\mathcal A$, and H\"older's inequality give,
\begin{align*}
d^{1/3}\E_\zeta[(\Delta\zeta)^2]
&=\frac{\varepsilon^4}{(\delta\kappa)^2}\E_\zeta[g^2\mathcal A(\Rcancel)]\\
&=\frac{1}{(\delta\kappa)^2}\E[g_0^2\mathcal A(R_0)]+O(\varepsilon)\\
&=h(\ell)+O(\varepsilon),
\end{align*}
where the final equality follows from \eqref{eq:K_moment_diffusion} and \eqref{eq:h_diffusion}. This proves \eqref{eq:conditional_variance_diffusion}. Finally, the conditional expansion and its moment bounds give, $\E_\zeta[g^4\mathds1_{\Omega_d}]=O(\varepsilon^{-8})$. Using \eqref{eq:delta_zeta}, it follows that
\begin{equation*}
\sup_{|\zeta|\le L} \E_\zeta[(\Delta\zeta)^4\mathds1_{\Omega_d}] \le
\frac{\varepsilon^{12}}{(\delta\kappa)^4} \sup_{|\zeta|\le L}\E_\zeta[g^4\mathds1_{\Omega_d}] =O(\varepsilon^4).
\end{equation*}
The complement is negligible by \eqref{eq:bad_event_increment_bound} and multiplication by $d^{1/3}=\varepsilon^{-2}$ proves \eqref{eq:conditional_fourth_diffusion}.
\end{pf}

\subsubsection{Transverse and mixed conditional moments}
\label{app:transverse_diffusion_moments}

Write $\E_{\zeta,X}$ for one-step expectation conditional on the current pair, and let $X'=\langle z'-b,e_\perp\rangle$ denote the proposed transverse coordinate before acceptance (so that $X_{k+1} = X_k + \mathds1_{\mathrm{acc}} (X'-X_k)$). All bounds below are uniform over $|\zeta|,|X|\le L$. Choose the second tangent basis vector along the component of $P_ze_{\perp}$ orthogonal to $P_z(b)$. Since $t=t^*+\zeta\varepsilon^3$ and $X$ is bounded, this component is nonzero for all sufficiently large $d$. We now specifically choose, for the orthonormal completion used in~\eqref{eq:orthonormal-basis}, the vector $e_2$ along the component of $P_z e_\perp$ orthogonal to $e_1$, i.e.~$e_2 = (P_z e_\perp - \langle P_z e_\perp,e_1\rangle e_1)/\|P_z e_\perp - \langle P_z e_\perp,e_1\rangle e_1\|$, where the remaining $e_i$, $i\in\{3,\ldots,d-1\}$ complete an orthonormal basis of $T_z \Mc$, with $\xi_i=\langle \xi, e_i\rangle$. Then

\begin{equation}
    \langle P_z\xi,e_{\perp}\rangle =-\frac{tX}{r\sqrt{1-t^2}}\,\xi_1 + \sqrt{1-\frac{X^2}{r^2(1-t^2)}}\,\xi_2, \qquad W=\xi_2^2+\widetilde W,
 \label{eq:joint_noise_decomp}
\end{equation}
where $\widetilde W\sim\chi^2_{d-3}$ and $\xi_1,\xi_2\sim\N(0,1)$ are mutually independent. In particular, the conditional law of $(\xi_1,W)$ is unchanged, so Proposition~\ref{prop:conditional_moments_diffusion} also holds with $\E_\zeta$ replaced by $\E_{\zeta,X}$. On a successful projection, $\langle P_z(b),e_{\perp}\rangle=-\delta tX/\kappa$ gives
\begin{equation}
    X'-X = \langle v,e_\perp\rangle + \alpha \langle \nhat,e_\perp\rangle =  \left(\frac{\alpha}{r}-\frac{\gamma\delta t}{\kappa}\right)X +\beta\langle P_z\xi,e_{\perp}\rangle.
 \label{eq:X_update_pcn}
\end{equation}
By \eqref{eq:alpha_r_conditional}, $\gamma=-\ell^2\varepsilon^2/2+O(\varepsilon^4)$ and $\delta t^*/\kappa=-\tau^2$, the coefficient of $X$ is $-s\ell^2\varepsilon^2/2+\widetilde A_d+O(\varepsilon^4)$ on $\Omega_d$. The two noise coefficients in front of $\xi_1$ and $\xi_2$ in~\eqref{eq:joint_noise_decomp} are $O(\varepsilon^3)$ and $1+O(\varepsilon^6)$, respectively. Hence, for every fixed $p\ge1$,
\begin{equation}
    X'-X=\ell\varepsilon \xi_2-\frac{s\ell^2}{2}\varepsilon^2X+\eta_d,
    \qquad
    \sup_{|\zeta|,|X|\le L}\E_{\zeta,X}[|\eta_d|^p\mathds1_{\Omega_d}] =O(\varepsilon^{4p}).
\label{eq:transverse_proposal_expansion}
\end{equation}
The transformation $\xi_2\mapsto-\xi_2$ leaves $(\xi_1,W)$, and hence $g$, $\Rcancel$ and $\Omega_d$ unchanged. Thus
\begin{equation}
\E_{\zeta,X}[H \xi_2\mathcal A(\Rcancel)\mathds1_{\Omega_d}]=0
\label{eq:transverse_sign_symmetry}
\end{equation}
for every integrable expression in which $H$ is a function of the state and $(\xi_1,W)$. Moreover, $R_0$ depends only on $\xi_1$, independently of $\xi_2$. The Lipschitz property of $\mathcal A$, Lemma~\ref{lem:conditional_expansion_diffusion} and Cauchy--Schwarz therefore give
\begin{equation}
\E_{\zeta,X}\bigl[\xi_2^p\,\mathcal A(\Rcancel)\,\mathds1_{\Omega_d}\bigr] = 2\Phi(-\sigma_\ell/2) +O(\varepsilon) ,\qquad p\in\{0,2\}.
\label{eq:transverse_acceptance_moments}
\end{equation}

Note that removing $\Omega_d$ from the corresponding expectations with $R_0$ is exponentially insignificant by the Gaussian moments and Lemma~\ref{lem:good_event}. Integrating out the acceptance uniform and applying \eqref{eq:transverse_proposal_expansion}--\eqref{eq:transverse_acceptance_moments} yields 
\begin{align}
    \E_{\zeta,X}[\Delta X] &=-\tfrac12 s h_\perp(\ell)X\varepsilon^2+O(\varepsilon^3),
\label{eq:transverse_drift}\\
    \E_{\zeta,X}[(\Delta X)^2] &= h_\perp(\ell)\varepsilon^2+O(\varepsilon^3),
\label{eq:transverse_variance}\\
    \E_{\zeta,X}[\Delta\zeta\,\Delta X] &= O(\varepsilon^3),
\label{eq:transverse_cross}\\
     \E_{\zeta,X}[(\Delta X)^4] &= O(\varepsilon^4).
\label{eq:transverse_fourth}
\end{align}

For the mixed moment, the leading $g\xi_2$ term vanishes by \eqref{eq:transverse_sign_symmetry}, the remaining terms are $O(\varepsilon^3)$ by \eqref{eq:conditional_variance_diffusion}, \eqref{eq:transverse_proposal_expansion} and H\"older's inequality. The fourth-moment bound follows directly from \eqref{eq:transverse_proposal_expansion}. All four estimates hold on $\Omega_d$, on its complement, $|\Delta X|\le2r$ and $|\Delta\zeta|\le2\sqrt d$, while its probability is exponentially small. Multiplication by $d^{1/3}=\varepsilon^{-2}$ gives the required transverse drift and variance limits, vanishing mixed covariance, and negligible large jumps.

\subsubsection{Completion of the process argument}

We use the following finite-dimensional diffusion criterion for a discrete chain accelerated by a sequence $c_d\to\infty$.

\begin{prop}[Ethier--Kurtz diffusion criterion,~\cite{ethier-kurtz:2009}]
\label{prop:EK_diffusion_criterion}
Let $c_d\to\infty$ be a positive sequence. For each $d$, let $(Y_k^{(d)})_{k\ge0}$ be a time-homogeneous Markov chain on $\RR^n$ for fixed $n$. Given $Y_k^{(d)}=x$, write $\Delta_d=Y_{k+1}^{(d)}-Y_k^{(d)}$ and define
\begin{equation}
    \mathsf b_d(x):=c_d\E_x[\Delta_d\mathds1_{\{|\Delta_d|\le1\}}]\,, \quad
    \mathsf a_d(x):=c_d\E_x[\Delta_d\Delta_d^\top \mathds1_{\{|\Delta_d|\le1\}}]\,, \quad
    \mathsf q_d^\eta(x):=c_d\mathbb P_x(|\Delta_d|\ge\eta)\,,
\end{equation}
with $\eta>0$. Assume the following:
\begin{enumerate}
    \item[(i)] for every $R,\eta>0$, the quantities $\mathsf b_d$, $\mathsf a_d$ and $\mathsf q_d^\eta$ converge uniformly on $\{|x|\le R\}$ to $b(x)$, $a(x)$ and zero respectively, where $b:\RR^n \to \RR^n$ and $a:\RR^n \to \RR^{n \times n}$ are continuous, with $a(x)$ symmetric and nonnegative definite;
    \item[(ii)] the $C([0,\infty),\RR^n)$ martingale problem for
    \begin{equation*}
        \mathcal LF(x)= \sum_{i=1}^n b_i(x)\partial_i F(x)+\tfrac12 \sum_{i,j=1}^n a_{ij}(x)\partial_{ij} F(x), \qquad F\in C_c^\infty(\RR^n),
    \end{equation*}
    has a unique solution for every initial law in $\mathcal P(\RR^n)$;
    \item[(iii)] there exists $\nu\in\mathcal P(\RR^n)$ such that $Y_0^{(d)}\myrightarrow{d}\nu$.
\end{enumerate}
Then, as $d\to\infty$,
\begin{equation*}
\bigl(Y^{(d)}_{\lfloor c_du\rfloor}\bigr)_{u\ge0}\;\Rightarrow\;Y \quad\text{in }D([0,\infty),\RR^n),
\end{equation*}
where $Y$ solves the martingale problem for $(\mathcal L,\nu)$.
\end{prop}

\begin{pf}
This is \citet[Chapter~7, Corollary~4.2]{ethier-kurtz:2009} with the integer clock $n$ replaced by a real sequence $c_d\to\infty$. The proof of this result uses only that the normalisation and the clock are the same positive number and not that this number is an integer and so this variant trivially follows by the same arguments.
\end{pf}

By Lemma~\ref{lem:scalar_markovianity}, $(\zeta_k,X_k)$ is a Markov chain. Its state space is
\begin{equation}
E_d=\left\{(\zeta,X):\ |t|\le1,\quad |X|\le r\sqrt{1-t^2},\quad t=t^*+\frac{\zeta}{\sqrt d}\right\},
\end{equation}
which contains every fixed compact subset of $\RR^2$ for all sufficiently large $d$. Extend the kernel by absorbing states outside $E_d$.

Take $c_d=d^{1/3}$. Proposition~\ref{prop:conditional_moments_diffusion} and \eqref{eq:transverse_drift}--\eqref{eq:transverse_fourth} give the untruncated conditional moments limits locally uniformly, with

\begin{equation}
b(\zeta,X)=\left(-\frac{h(\ell)}{2\sigma_{t^*}^2}\zeta,\, -\frac{s}{2}h_\perp(\ell)X\right),
\qquad
a=\begin{pmatrix}h(\ell)&0\\0&h_\perp(\ell)\end{pmatrix}.
\end{equation}
Since $|(\Delta\zeta,\Delta X)|^4\le2((\Delta\zeta)^4+(\Delta X)^4)$, we also have $c_d\E_{\zeta,X}|\Delta|^4\to0$ locally uniformly. This bounds the truncation errors in $\mathsf b_d$ and $\mathsf a_d$, and bounds $\mathsf q_d^\eta$ after multiplication by $\eta^{-4}$. Condition~(i) of Proposition~\ref{prop:EK_diffusion_criterion} follows. The limiting coefficients are globally Lipschitz so its martingale problem is well posed, giving condition~(ii).

For the initial law, conditional on $t$, write $X=r\sqrt{1-t^2}V_\perp$, where $V$ is uniform on the unit sphere in $b^\perp$ and independent of $t$ with $V_\perp=\langle V,e_{\perp}\rangle$. Writing $V=G/\|G\|$ with $G\sim\N(0,I_{d-1})$ gives $\sqrt{d-1}V_\perp\myrightarrow{d}\N(0,1)$. Also $r^2(1-t^2)/(d-1)\myrightarrow{p}s^{-1}$ by Lemma~\ref{lem:concentration_t} and \eqref{eq:s_two_forms}. Independence of $V$ and $t$ and Slutsky's theorem yields
\begin{equation}
(\zeta_0,X_0)\myrightarrow{d} \N(0,\sigma_{t^*}^2)\otimes\N(0,s^{-1}).
\label{eq:pair_initial_law}
\end{equation}
Thus condition~(iii) holds. Proposition~\ref{prop:EK_diffusion_criterion} with $n=2$ proves \eqref{eq:joint_process_convergence}. The diffusion matrix, separate linear drifts and product initial law identify the limit as the independent processes \eqref{eq:scalar_OU_limit} and \eqref{eq:transverse_OU_limit}. Their initial variances are invariant, so both processes are stationary.

\subsection{Proof of Corollary~\ref{cor:scaling_trichotomy}}
\label{app:scaling_trichotomy_proof}

By Lemma~\ref{lem:g_cancellation}, the expansion \eqref{eq:g_final_clean} holds for every $\beta_d\to0$. We use the decomposition $\Rcancel=Jg+Q$ of \eqref{eq:R_JQ_decomposition}. From \eqref{eq:g_final_clean}, since the terms $\beta_d$, $\beta_d^2\sqrt d$, $\beta_d^3\sqrt d$ and $\beta_d^6 d$ are dominated by $\beta_d\sqrt d$ and $\beta_d^4 d$ (for $\beta_d\le1$) so,
\begin{equation}
  g=O_p\,\bigl(\beta_d\sqrt d+\beta_d+\beta_d^2\sqrt d+\beta_d^3\sqrt d
              +\beta_d^4 d\bigr)
   =O_p\,\bigl(\beta_d\sqrt d\,(1+\lambda_d)\bigr).
  \label{eq:g_orderbound}
\end{equation}
With $r=\kappa\sqrt d$, we have $g^2 /r^2 = O_p\,\bigl(\beta_d^2(1+\lambda_d^2)\bigr)$ so, as $(1-\rho)/(1+\rho)=O(\beta_d^2)$ and $\beta_d\to0$,
\begin{equation*}
  Q=\frac{1-\rho}{1+\rho}\frac{g^2}{2r^2}
   =O_p\,\bigl(\beta_d^4(1+\lambda_d^2)\bigr)
   =o_p(\lambda_d+\lambda_d^2),
\end{equation*}
the last equality because $\beta_d^4/\lambda_d=\beta_dd^{-1/2}\to0$ and $\beta_d^4\lambda_d^2/\lambda_d^2=\beta_d^4\to0$. By Lemma~\ref{lem:coeff_general} (with $c_\beta=\gamma$) and $C_*\neq0$, $J=C_*\beta_d^2(1+o_p(1))$, so
\begin{equation*}
  Jg=C_*\beta_d^2 g+o_p(1)\cdot C_*\beta_d^2 g
    =C_*\beta_d^2 g+o_p(\lambda_d+\lambda_d^2),
\end{equation*}
the remainder being the product of $o_p(1)$ with $C_*\beta_d^2 g=O_p(\lambda_d+\lambda_d^2)$. Multiplying \eqref{eq:g_final_clean} by $C_*\beta_d^2$ and using $\beta_d^3\sqrt d=\lambda_d$ produces the two leading terms $C_*\tau\lambda_d\xi_1$ and $C_*\Lambda\lambda_d^2$, while each remaining product is $o_p(\lambda_d+\lambda_d^2)$:
\begin{align*}
  C_*\beta_d^2\,O_p(\beta_d) &= O_p(\beta_d^3)= O_p\bigl(\lambda_d\, d^{-1/2}\bigr),\\
  C_*\beta_d^2\,O_p(\beta_d^2\sqrt d) &= O_p(\beta_d\lambda_d),\\
  C_*\beta_d^2\,O_p(\beta_d^3\sqrt d) &= O_p(\beta_d^2\lambda_d),\\
  C_*\beta_d^2\,O(\beta_d^6 d) &= O(\beta_d^2\lambda_d^2),
\end{align*}
each of which is $o_p(\lambda_d)$ or $o_p(\lambda_d^2)$ because
$\beta_d\to0$ and $d\to\infty$. Combining the bounds gives, for every $\beta_d\to0$,
\begin{equation}
\Rcancel=C_*\tau\,\lambda_d\,\xi_1+C_*\Lambda\,\lambda_d^2+o_p\,\bigl(\lambda_d+\lambda_d^2\bigr).
  \label{eq:Rcancel_uniform}
\end{equation}
The remainder is thus negligible relative to the retained terms in every regime and the trichotomy follows by inspection.

\paragraph{Subcritical, $\lambda_d\to0$.} We have that $\lambda_d+\lambda_d^2\to0$, so all three terms of \eqref{eq:Rcancel_uniform} tend to $0$ in probability and $\Rcancel\myrightarrow{p}0$. Since $0\le1 \wedge e^{\Rcancel}\le1$ and the integrand $ \myrightarrow{p} 1$, boundedness implies convergence in $L^1$, and hence $a_d(\beta_d)\to1$.

\paragraph{Critical, $\lambda_d\to\ell^3\in(0,\infty)$.} Then $\lambda_d+\lambda_d^2\to\ell^3+\ell^6\in(0,\infty)$, so the remainder is $o_p(1)$ and \eqref{eq:Rcancel_uniform} reads
$\Rcancel=C_*\tau\ell^3\xi_1+C_*\Lambda\ell^6+o_p(1)$, and the conclusion follows from Theorem~\ref{thm:gaussian_limit} and Corollary~\ref{cor:acceptance_limit}.

\paragraph{Supercritical, $\lambda_d\to\infty$.} Dividing
\eqref{eq:Rcancel_uniform} by $\lambda_d^2$ (so that $o_p(\lambda_d+\lambda_d^2)/\lambda_d^2 =o_p(\lambda_d^{-1}+1)=o_p(1)$),
\begin{equation*}
  \frac{\Rcancel}{\lambda_d^2} =\frac{C_*\tau\,\xi_1}{\lambda_d}+C_*\Lambda+o_p(1)
  \myrightarrow{p} C_*\Lambda=-\tfrac12 C_*^2\tau^2<0.
\end{equation*}
Since $\lambda_d^2\to\infty$, it follows that $\Rcancel \myrightarrow{p} -\infty$, so $1  \wedge e^{\Rcancel} \myrightarrow{p} 0$. Again, boundedness implies convergence in $L^1$, and therefore $a_d(\beta_d)\to0$.

\subsection{Proof of Theorem~\ref{thm:fixed_step}}
\label{app:thm:fixed_step_proof}

Take a fixed $\beta\in(0,1)$ and have $J_s$ denote the coefficient of $g$ in \eqref{eq:Rcancel_intrinsic}. We use the identities $\beta^2=-2\gamma-\gamma^2$ and $(\delta /\kappa)t^*=-\tau^2$ throughout. Using the expansions of $\|P_z(b)\|$ and $\langle\nhat,b\rangle$, the proposal \eqref{eq:intrinsic_proposal} and its normal projection give
\begin{align}
\frac{\|v\|^2}{d}
  &= \left(\frac{\gamma\delta}{s}\sqrt{1-t^2} +\frac{\beta}{\sqrt s}\frac{\xi_1}{\sqrt d}\right)^2
  +\frac{\beta^2}{s}\frac Wd,\\
    \frac{\alpha}{r}
  &=-1+\sqrt{1-\frac{\|v\|^2}{\kappa^2d}},\\
  \frac gd
  &=
  \frac{\gamma\delta^2}{s}(1-t^2)+\frac{\beta\delta}{\sqrt s}\sqrt{1-t^2} \frac{\xi_1}{\sqrt d}
  +\frac{\alpha}{r}\delta\kappa t,\\
  J_s&=
  \frac{1-\rho}{1+\rho}+\frac{\alpha/r}{1+\rho}+\frac{1-\rho}{1+\rho}\frac{\delta t}{\kappa s}.
\label{eq:three_scalars}
\end{align}
At $(t,\xi_1/\sqrt d,W/d)=(t^*,0,1)$, i.e. where the three statistics sit in the limit, we have
\begin{equation}
1-\frac{1}{\kappa^2}\left(\frac{\gamma^2\tau^2}{s^2} +\frac{\beta^2}{s}\right)
  =\left(1+\frac{\gamma}{s\kappa^2}\right)^2=:q^2.
\end{equation}
Since $-1<\gamma<0$ and $s\kappa^2>1$, we have $0<q<1$. At this same point, $\alpha/r=q -1=\gamma/(s\kappa^2)$ and
\begin{equation}
\frac gd=\frac{\gamma}{s}\left(\tau^2+\frac{\delta}{\kappa}t^*\right)=0,
  \qquad
  J_s=\frac{\gamma}{1+\rho}\left(-1+\frac{\tau^2+\kappa^{-2}}s\right)=0.
\end{equation}
Thus the scalar functions in \eqref{eq:three_scalars} are smooth in a neighbourhood of this point. Stationary concentration, $\xi_1/\sqrt d\to0$ and $W/d\to1$ imply $\|v\|^2 / r^2\myrightarrow{p}1-q^2<1$. The projection therefore fails with probability tending to zero.

Lemma~\ref{lem:concentration_t} and the $\chi^2$ CLT give
\begin{equation}
\sqrt d \left(t-t^*,\,\frac{\xi_1}{\sqrt d},\,\frac Wd-1\right) \myrightarrow{d} (\zeta,\xi_1,\sqrt2\,\omega)
\end{equation}
where the limiting variables are independent and have the Gaussian laws given in the theorem. Applying the multivariate delta method to \eqref{eq:three_scalars}, using the two cancellations above, yields jointly
\begin{equation}
\left(\frac g{\sqrt d},\,\sqrt d\,J_s\right) \myrightarrow{d}(\bar g,\bar J_s).
\label{eq:joint_limit}
\end{equation}
Differentiation gives the coefficients in \eqref{eq:gbar} and \eqref{eq:Jbar} and their simpler forms uses $q s+\gamma\tau^2=s\rho$ and the identities stated at the beginning of the proof.

Since $r^2=\kappa^2d$, the exact log acceptance ratio can be written as
\begin{equation}
  \Rcancel =(\sqrt d\,J_s)\frac g{\sqrt d}
    +\frac{1-\rho}{1+\rho}\frac{1}{2s\kappa^2} \left(\frac g{\sqrt d}\right)^2\\
\myrightarrow{d}
    \bar J_s\bar g +\frac{1-\rho}{1+\rho}\frac{\bar g^2}{2s\kappa^2}=\bar R.
\end{equation}
These convergences hold jointly, with failed projections resulting in an event of vanishing probability. Assigning zero acceptance to failed projections, boundedness and continuity of $x\mapsto 1 \wedge e^x$ give \eqref{eq:fixed_step_acceptance}. The limit is positive because $\bar R$ is almost surely finite. To see that it is strictly below one, evaluate it at $(\zeta,\xi_1,\omega)=(0,0,1)$ which gives
\begin{equation}
    \bar R = \frac{\gamma\beta^2\tau^2}{2s^2\kappa^2q^2} \left(1+\frac{\gamma\tau^2}{2s}\right)<0.
\end{equation}
Indeed, $\gamma<0$ and the parenthesis is greater than $1/2$. Continuity and full support of the Gaussian law imply $\mathbb P(\bar R<0)>0$, hence $a(\beta)<1$. This proves part~(i).

For part~(ii), write $\zeta',X'$ for the proposed coordinates. On a successful projection we have,
\begin{align}
  \zeta'&=\zeta+\frac{g}{\delta\kappa\sqrt d},\label{eq:zeta_update}\\
  X' &=
  \left(1+\frac{\alpha}{r}-\frac{\gamma\delta t}{s\kappa}\right)X
  +\frac{\beta}{\sqrt s}\langle P_z\xi,e_{\perp}\rangle.\label{eq:X_update}
\end{align}
The second identity follows from $\langle P_z(b),e_{\perp}\rangle=-\delta tX/\kappa$. Its coefficient of $X$ converges to $q+\gamma\tau^2/s=\rho$. The joint noise decomposition \eqref{eq:joint_noise_decomp} depends only on the tangent geometry and applies unchanged to this proposal. The coefficients of $\xi_1$ and $\xi_2$ in \eqref{eq:joint_noise_decomp} converge to 0 and 1 respectively, as $d \to \infty$, uniformly on compact sets of $(\zeta,X)$. Since $\xi_2^2/\sqrt d\myrightarrow{p}0$, the $\chi^2$ CLT gives
\begin{equation}
  \left(\xi_1,\langle P_z\xi,e_{\perp}\rangle,\frac{W-d}{\sqrt{2d}} \right) \myrightarrow{d}(\xi_1,\xi_2, \omega),
\label{eq:joint_noise_limit}
\end{equation}
where $\xi_1,\xi_2,\omega\sim\N(0,1)$ are independent. The Taylor expansions used to obtain \eqref{eq:joint_limit} have remainders that vanish in probability uniformly on compact sets of $(\zeta,X)$. Smoothness gives uniform remainder bounds when the noise fluctuations are bounded, and tightness controls the probability of larger fluctuations. Together with \eqref{eq:joint_noise_limit}, we have
\begin{equation}
  \left(\zeta',X',\Rcancel\right) \myrightarrow{d}
  \left(\zeta+\frac{\bar g}{\delta\kappa},\rho X+\frac{\beta}{\sqrt s}\xi_2,\bar R \right)\,.
\end{equation}
The one-step transition laws therefore converge to \eqref{eq:limit_kernel} with the same uniformity, by continuity. Lemma~\ref{lem:scalar_markovianity} gives Markovianity of the pair and its stationary initial law converges to the product Gaussian law in \eqref{eq:pair_initial_law}. The limiting kernel maps bounded continuous functions to bounded continuous functions. Together with convergence of the initial laws, locally uniform convergence of the transition kernels and tightness of the stationary marginals, gives convergence over every fixed number of iterations by induction; see the proof of \cite[Chapter~4, Lemma~8.1]{ethier-kurtz:2009}. Finally, passing to the limit in the detailed balance condition for the finite-$d$ chains shows that the limiting kernel is reversible with respect to $\N(0,\sigma_{t^*}^2)\otimes\N(0,s^{-1})$. This completes the proof of part~(ii).

\bibliographystyle{unsrtnat}
\bibliography{bibliography}

@Article{	  ben-arous:1988,
  title		= {{Methods de {L}aplace et de la phase stationnaire sur
		  l'espace de {W}iener}},
  volume	= {25},
  issn		= {0090-9491},
  doi		= {10.1080/17442508808833536},
  number	= {3},
  journal	= {Stochastics},
  author	= {Ben Arous, G\'erard},
  month		= dec,
  year		= {1988},
  pages		= {125--153}
}

@Article{	  beskos-girolami-lan-etal:2017,
  title		= {Geometric {MCMC} for infinite-dimensional inverse
		  problems},
  volume	= {335},
  issn		= {0021-9991},
  doi		= {10.1016/j.jcp.2016.12.041},
  journal	= {Journal of Computational Physics},
  author	= {Beskos, Alexandros and Girolami, Mark and Lan, Shiwei and
		  Farrell, Patrick E. and Stuart, Andrew M.},
  month		= apr,
  year		= {2017},
  pages		= {327--351}
}

@Article{	  beskos-pillai-roberts-etal:2013,
  title		= {Optimal tuning of the hybrid {Monte} {Carlo} algorithm},
  volume	= {19},
  issn		= {1350-7265},
  doi		= {10.3150/12-BEJ414},
  number	= {5A},
  journal	= {Bernoulli},
  author	= {Beskos, Alexandros and Pillai, Natesh and Roberts, Gareth
		  and Sanz-Serna, Jesus-Maria and Stuart, Andrew},
  month		= nov,
  year		= {2013},
  pages		= {1501--1534}
}

@Article{	  byrne-girolami:2013,
  title		= {Geodesic {Monte} {Carlo} on {Embedded} {Manifolds}},
  volume	= {40},
  copyright	= {© 2013 The Authors. Scandinavian Journal of Statistics
		  published by John Wiley \& Sons Ltd on behalf of The Board
		  of the Foundation of the Scandinavian Journal of
		  Statistics.},
  issn		= {1467-9469},
  doi		= {10.1111/sjos.12036},
  number	= {4},
  journal	= {Scandinavian Journal of Statistics},
  author	= {Byrne, Simon and Girolami, Mark},
  year		= {2013},
  pages		= {825--845}
}

@Article{	  cotter-roberts-stuart-etal:2013,
  title		= {{MCMC} {Methods} for {Functions}: {Modifying} {Old}
		  {Algorithms} to {Make} {Them} {Faster}},
  volume	= {28},
  issn		= {0883-4237, 2168-8745},
  shorttitle	= {{MCMC} {Methods} for {Functions}},
  doi		= {10.1214/13-STS421},
  number	= {3},
  journal	= {Statistical Science},
  author	= {Cotter, S. L. and Roberts, G. O. and Stuart, A. M. and
		  White, D.},
  month		= aug,
  year		= {2013},
  pages		= {424--446}
}

@InCollection{	  diaconis-holmes-shahshahani:2013,
  title		= {Sampling from a {Manifold}},
  volume	= {10},
  booktitle	= {Advances in {Modern} {Statistical} {Theory} and
		  {Applications}: {A} {Festschrift} in honor of {Morris} {L}.
		  {Eaton}},
  publisher	= {Institute of Mathematical Statistics},
  author	= {Diaconis, Persi and Holmes, Susan and Shahshahani,
		  Mehrdad},
  month		= jan,
  year		= {2013},
  doi		= {10.1214/12-IMSCOLL1006},
  pages		= {102--126}
}

@Book{		  frenkel-smit:2002,
  address	= {San Diego},
  edition	= {2nd ed},
  series	= {Computational science series},
  title		= {Understanding molecular simulation: from algorithms to
		  applications},
  isbn		= {978-0-08-051998-2 978-0-12-267351-1},
  shorttitle	= {Understanding molecular simulation},
  number	= {1},
  publisher	= {Academic Press},
  author	= {Frenkel, Daan and Smit, Berend},
  year		= {2002}
}

@Article{	  girolami-calderhead:2011,
  title		= {Riemann {Manifold} {Langevin} and {Hamiltonian} {Monte}
		  {Carlo} {Methods}},
  volume	= {73},
  copyright	= {https://academic.oup.com/journals/pages/open\_access/funder\_policies/chorus/standard\_publication\_model},
  issn		= {1369-7412, 1467-9868},
  doi		= {10.1111/j.1467-9868.2010.00765.x},
  number	= {2},
  journal	= {Journal of the Royal Statistical Society Series B:
		  Statistical Methodology},
  author	= {Girolami, Mark and Calderhead, Ben},
  month		= mar,
  year		= {2011},
  pages		= {123--214}
}

@Article{	  habeck-hasenpflug-kodgirwar-etal:2025,
  title		= {Geodesic slice sampling on the sphere},
  author	= {Habeck, Michael and Hasenpflug, Mareike and Kodgirwar,
		  Shantanu and Rudolf, Daniel},
  journal	= {Journal of Machine Learning Research},
  volume	= {26},
  number	= {297},
  pages		= {1--38},
  year		= {2025}
}

@Article{	  hartmann:2008,
  title		= {An {Ergodic} {Sampling} {Scheme} for {Constrained}
		  {Hamiltonian} {Systems} with {Applications} to {Molecular}
		  {Dynamics}},
  volume	= {130},
  issn		= {1572-9613},
  doi		= {10.1007/s10955-007-9470-2},
  number	= {4},
  journal	= {Journal of Statistical Physics},
  author	= {Hartmann, Carsten},
  month		= feb,
  year		= {2008},
  pages		= {687--711}
}

@InProceedings{	  ho-jain-abbeel:2020,
  title		= {Denoising Diffusion Probabilistic Models},
  author	= {Ho, Jonathan and Jain, Ajay and Abbeel, Pieter},
  booktitle	= {Advances in Neural Information Processing Systems},
  volume	= {33},
  pages		= {6840--6851},
  year		= {2020}
}

@Article{	  inahama:2014,
  title		= {Malliavin differentiability of solutions of rough
		  differential equations},
  volume	= {267},
  issn		= {00221236},
  doi		= {10.1016/j.jfa.2014.06.011},
  number	= {5},
  journal	= {Journal of Functional Analysis},
  author	= {Inahama, Yuzuru},
  month		= sep,
  year		= {2014},
  pages		= {1566--1584}
}

@InProceedings{	  kidger-foster-li-etal:2021,
  title		= {Neural SDEs as Infinite-Dimensional GANs},
  author	= {Kidger, Patrick and Foster, James and Li, Xuechen and
		  Lyons, Terry J},
  booktitle	= {Proceedings of the 38th International Conference on
		  Machine Learning},
  pages		= {5453--5463},
  year		= {2021},
  editor	= {Meila, Marina and Zhang, Tong},
  volume	= {139},
  series	= {Proceedings of Machine Learning Research},
  month		= {18--24 Jul},
  publisher	= {PMLR}
}

@Article{	  klose:2026,
  title		= {The constant coefficient in precise {Laplace} asymptotics
		  for {gPAM}},
  volume	= {62},
  issn		= {0246-0203},
  doi		= {10.1214/24-AIHP1548},
  number	= {2},
  journal	= {Annales de l'Institut Henri Poincaré, Probabilités et
		  Statistiques},
  author	= {Klose, Tom},
  month		= may,
  year		= {2026}
}

@Article{	  lelievre-rousset-stoltz:2012,
  title		= {Langevin dynamics with constraints and computation of free
		  energy differences},
  volume	= {81},
  issn		= {0025-5718, 1088-6842},
  doi		= {10.1090/S0025-5718-2012-02594-4},
  number	= {280},
  journal	= {Mathematics of Computation},
  author	= {Leli\`evre, Tony and Rousset, Mathias and Stoltz, Gabriel},
  month		= oct,
  year		= {2012},
  pages		= {2071--2125}
}

@Article{	  lelievre-rousset-stoltz:2019,
  title		= {Hybrid {Monte} {Carlo} methods for sampling probability
		  measures on submanifolds},
  volume	= {143},
  issn		= {0945-3245},
  doi		= {10.1007/s00211-019-01056-4},
  number	= {2},
  journal	= {Numerische Mathematik},
  author	= {Leli\`evre, Tony and Rousset, Mathias and Stoltz, Gabriel},
  month		= oct,
  year		= {2019},
  pages		= {379--421}
}

@Article{	  lie-rudolf-sprungk-etal:2023,
  title		= {Dimension-independent {Markov} chain {Monte} {Carlo} on
		  the sphere},
  volume	= {50},
  copyright	= {© 2023 The Authors. Scandinavian Journal of Statistics
		  published by John Wiley \& Sons Ltd on behalf of The Board
		  of the Foundation of the Scandinavian Journal of
		  Statistics.},
  issn		= {1467-9469},
  doi		= {10.1111/sjos.12653},
  number	= {4},
  journal	= {Scandinavian Journal of Statistics},
  author	= {Lie, Han Cheng and Rudolf, Daniel and Sprungk, Bj\"orn and
		  Sullivan, T. J.},
  year		= {2023},
  pages		= {1818--1858}
}

@InProceedings{	  lipman-chen-ben-hamu-etal:2023,
  title		= {Flow Matching for Generative Modeling},
  author	= {Yaron Lipman and Ricky T. Q. Chen and Heli Ben-Hamu and
		  Maximilian Nickel and Matthew Le},
  booktitle	= {The Eleventh International Conference on Learning
		  Representations },
  year		= {2023}
}

@Book{		  mardia-jupp:2000,
  address	= {Chichester ; New York},
  series	= {Wiley series in probability and statistics},
  title		= {Directional statistics},
  isbn		= {978-0-471-95333-3},
  publisher	= {J. Wiley},
  author	= {Mardia, K. V. and Jupp, Peter E.},
  year		= {2000}
}

@Book{		  nualart:2006,
  address	= {Berlin ; New York},
  edition	= {2nd ed},
  series	= {Probability and its applications},
  title		= {The {Malliavin} calculus and related topics},
  isbn		= {978-3-540-28328-7},
  publisher	= {Springer},
  author	= {Nualart, David},
  year		= {2006}
}

@Article{	  roberts-gelman-gilks:1997,
  title		= {Weak convergence and optimal scaling of random walk
		  {Metropolis} algorithms},
  volume	= {7},
  issn		= {1050-5164},
  doi		= {10.1214/aoap/1034625254},
  number	= {1},
  journal	= {The Annals of Applied Probability},
  author	= {Roberts, G. O. and Gelman, A. and Gilks, W. R.},
  month		= feb,
  year		= {1997}
}

@Article{	  roberts-rosenthal:1998,
  title		= {Optimal {Scaling} of {Discrete} {Approximations} to
		  {Langevin} {Diffusions}},
  volume	= {60},
  copyright	= {https://academic.oup.com/journals/pages/open\_access/funder\_policies/chorus/standard\_publication\_model},
  issn		= {1369-7412, 1467-9868},
  doi		= {10.1111/1467-9868.00123},
  number	= {1},
  journal	= {Journal of the Royal Statistical Society Series B:
		  Statistical Methodology},
  author	= {Roberts, Gareth O. and Rosenthal, Jeffrey S.},
  month		= jan,
  year		= {1998},
  pages		= {255--268}
}

@Article{	  schorlepp-grafke:2026,
  title		= {Scalability of the {Second}-{Order} {Reliability} {Method}
		  for {Stochastic} {Differential} {Equations} with
		  {Multiplicative} {Noise}},
  volume	= {14},
  issn		= {2166-2525},
  doi		= {10.1137/25M1743867},
  number	= {3},
  journal	= {SIAM/ASA Journal on Uncertainty Quantification},
  author	= {Schorlepp, Timo and Grafke, Tobias},
  month		= sep,
  year		= {2026},
  pages		= {920--973}
}

@Article{	  schorlepp-tong-grafke-etal:2023,
  author	= {Schorlepp, Timo and Tong, Shanyin and Grafke, Tobias and
		  Stadler, Georg},
  title		= {{Scalable methods for computing sharp extreme event
		  probabilities in infinite-dimensional stochastic systems}},
  journal	= {Stat. Comput.},
  year		= {2023},
  month		= {Oct},
  day		= {13},
  volume	= {33},
  number	= {6},
  pages		= {137},
  issn		= {1573-1375},
  doi		= {10.1007/s11222-023-10307-2}
}

@InProceedings{	  song-sohl-dickstein-kingma-etal:2021,
  title		= {Score-Based Generative Modeling through Stochastic
		  Differential Equations},
  author	= {Song, Yang and Sohl-Dickstein, Jascha and Kingma, Diederik
		  P. and Kumar, Abhishek and Ermon, Stefano and Poole, Ben},
  booktitle	= {International Conference on Learning Representations},
  year		= {2021}
}

@Article{	  stuart:2010,
  title		= {{Inverse problems: a Bayesian perspective}},
  author	= {Stuart, Andrew M},
  journal	= {Acta Numer.},
  volume	= {19},
  pages		= {451--559},
  year		= {2010},
  publisher	= {Cambridge University Press},
  doi		= {10.1017/S0962492910000061 }
}

@Article{	  wood:1994,
  title		= {Simulation of the von mises fisher distribution},
  volume	= {23},
  issn		= {0361-0918},
  doi		= {10.1080/03610919408813161},
  number	= {1},
  journal	= {Communications in Statistics - Simulation and
		  Computation},
  author	= {Wood, Andrew T.A},
  month		= jan,
  year		= {1994},
  pages		= {157--164}
}

@Article{	  zappa-holmes-cerfon-goodman:2018,
  title		= {Monte {Carlo} on {Manifolds}: {Sampling} {Densities} and
		  {Integrating} {Functions}},
  volume	= {71},
  copyright	= {© 2018 Wiley Periodicals, Inc.},
  issn		= {1097-0312},
  shorttitle	= {Monte {Carlo} on {Manifolds}},
  doi		= {10.1002/cpa.21783},
  number	= {12},
  journal	= {Communications on Pure and Applied Mathematics},
  author	= {Zappa, Emilio and Holmes-Cerfon, Miranda and Goodman,
		  Jonathan},
  year		= {2018},
  pages		= {2609--2647}
}

@book{ethier-kurtz:2009,
  title={Markov processes: characterization and convergence},
  author={Ethier, Stewart N and Kurtz, Thomas G},
  year={2009},
  publisher={John Wiley \& Sons}
}

@article{christensen-roberts-rosenthal:2005,
  title={Scaling limits for the transient phase of local Metropolis--Hastings algorithms},
  author={Christensen, Ole F and Roberts, Gareth O and Rosenthal, Jeffrey S},
  journal={Journal of the Royal Statistical Society Series B: Statistical Methodology},
  volume={67},
  number={2},
  pages={253--268},
  year={2005},
  publisher={Oxford University Press}
}

@article{jourdain-lelievre-miasojedow:2014,
 ISSN = {13507265},
 URL = {http://www.jstor.org/stable/43590428},
 author = {BENJAMIN JOURDAIN and TONY LELIÈVRE and BŁAŻEJ MIASOJEDOW},
 journal = {Bernoulli},
 number = {4},
 pages = {1930--1978},
 publisher = {[Bernoulli Society for Mathematical Statistics and Probability, International Statistical Institute (ISI)]},
 title = {Optimal scaling for the transient phase of Metropolis Hastings algorithms: The longtime behavior},
 urldate = {2026-09-30},
 volume = {20},
 year = {2014}
}

\end{document}